\documentclass[12pt]{amsart}
\usepackage{amssymb,amsthm,amsmath,amstext}
\usepackage{mathrsfs}  
\usepackage{bm}        
\usepackage{mathtools} 

\usepackage[left=1.28in,top=.8in,right=1.28in,bottom=.8in]{geometry}

\usepackage{graphicx}
\usepackage[all]{xy}
\usepackage[title]{appendix} 
\theoremstyle{plain}
\newtheorem{theorem}{Theorem}[section]
\newtheorem{prop}[theorem]{Proposition}
\newtheorem{lemma}[theorem]{Lemma}

\newtheorem{cor}[theorem]{Corollary}
\newcommand{\Brp}{{}_p\mathrm{Br}}

\newtheorem*{conj}{Conjecture}

\theoremstyle{definition}
\newtheorem{defin}[theorem]{Definition}

\theoremstyle{remark}
\newtheorem{remark}[theorem]{Remark}

\newcommand{\sheaf}[1]{\mathscr{#1}}
\newcommand{\LL}{\sheaf{L}}

\newcommand{\OO}{\sheaf{O}}

\newcommand{\EE}{\sheaf{E}}
\newcommand{\FF}{\sheaf{F}}

\newcommand{\CC}{\sheaf{C}}
\newcommand{\PP}{\sheaf{P}}

\newcommand{\Brtwo}{{}_2\mathrm{Br}}
\newcommand{\Pictwo}{{}_2\mathrm{Pic}}
\newcommand{\rPic}{{}_r\mathrm{Pic}}

\DeclareMathOperator{\End}{\mathrm{End}}

\newcommand{\Group}[1]{\mathbf{#1}}
\newcommand{\GL}{\Group{GL}}
\newcommand{\SL}{\Group{SL}}
\newcommand{\PGL}{\Group{PGL}}

\newcommand{\SOrth}{\Group{SO}}

\newcommand{\PSOrth}{\Group{PSO}}
\newcommand{\Sim}{\Group{Sim}}
\newcommand{\PSim}{\Group{PSim}}

\newcommand{\Z}{\mathbb Z}

\newcommand{\A}{\mathbb A}
\newcommand{\C}{\mathbb C}
\renewcommand{\P}{\mathbb P}

\DeclareMathOperator{\AAut}{\mathbf{Aut}}
\DeclareMathOperator{\Aut}{\mathrm{Aut}}
\DeclareMathOperator{\Br}{\mathrm{Br}}

\DeclareMathOperator{\Pic}{\mathrm{Pic}}
\DeclareMathOperator{\nPic}{\mathrm{_nPic}}
\DeclareMathOperator{\twoPic}{\mathrm{_2Pic}}

\newcommand{\gG}{\Group{G}}

\newcommand{\Gm}{\Group{G}_{\text{m}}}

\usepackage{hyperref}
\usepackage{color}

  \usepackage[OT2, T1]{fontenc}
\DeclareSymbolFont{cyrletters}{OT2}{wncyr}{m}{n}
\DeclareMathSymbol{\Sha}{\mathalpha}{cyrletters}{"58}

\hypersetup{colorlinks=true,linkcolor=blue,anchorcolor=blue,citecolor=blue,letterpaper=true}

\begin{document}

\title[ Hyperelliptic curves]
{Period-index problem for hyperelliptic curves}

\author[Jaya]{J. N. \ Iyer}
\address{Institute of Mathematical Sciences \\ %
 C.I.T. Campus \\ %
  Taramani, Chennai 600113. India}
\email{jniyer@imsc.res.in}

\author[Parimala]{R.\ Parimala  \smallskip \\
\MakeLowercase{with an appendix by} S. Ramanan}
\address{Department of Mathematics \& Computer Science \\ %
Emory University \\ %
400 Dowman Drive~NE \\ %
Atlanta, GA 30322, USA}
\email{praman@emory.edu}

\address{Chennai Mathematical Institute \\
H1, SIPCOT IT Park,  \\
Siruseri Kelambakkam 603103
India}
\email{sramanan@cmi.ac.in}

\date{}

\begin{abstract} 
Let $C$ be a smooth projective curve of genus 2 over a number field $k$ with a rational point. 
We prove that the index and exponent coincide for elements in the 2-torsion of $\Sha(Br(C))$.  
 In the appendix,  an  isomorphism of the moduli space of rank 2 stable vector bundles with odd determinant 
 on a smooth projective hyperelliptic curve $C$ of genus $g$ with a rational point over any field of characteristic not two with the  Grassmannian of $(g-1)$-dimensional  linear subspaces in the base locus of a certain pencil of quadrics is established, making a result of (\cite{De-Ra}) rational. We establish  a twisted version  of this isomorphism and we derive as a consequence a weak Hasse principle for the smooth intersection $X$  of two quadrics in ${\mathbb P}^5$ over a number field:  if $X$ contains a line locally, then $X$ has a $k$-rational point. 
   \end{abstract} 
 
\maketitle

\date{}

Let $k$ be a field and $\Br(k)$  the  Brauer group of $k$. 
 There are two numerical 
invariants attached to a Brauer class $\alpha \in \Br(k)$;  $index(\alpha) = \sqrt{[D : k]}$ if 
$\alpha$ is represented by a central division algebra $D$ over $k$ and $period(\alpha) = $
order of $\alpha$ in $\Br(k)$. There has been extensive study of uniform bounds for index of algebras
in terms of their periods over fields which are of arithmetic or geometric interest (\cite{DJ}, \cite{Lb},
\cite{saltman}, \cite{HHK}). 
Let $p$ be a prime and $K$ a field of characteristic not equal to $p$.  The Brauer $p$-dimension of $K$
denoted by $\Br\! _p$dim$(K)$ is the least integer $d$ such that 
for every finite extension $L/K$ and every $\alpha \in ~_p\Br(L)$, index$(\alpha)$ divides $p^d$. 
Here $\Brp(L)$ denotes the $p$-torsion subgroup of $\Br(L)$. 

Bounding Brauer dimension  has deep consequences in the study of homogeneous spaces under 
connected  linear algebraic groups. A theorem of de-Jong/Lieblich that period =  index for 
function fields of surfaces over algebraically closed fields is critical  to the solution  of Conjecture II of 
Serre for exceptional groups of type $D_4$, $E_6$, $E_7$ due to Gille (\cite[IV.2]{G}). 
Bounding the Brauer $2$-dimension of function fields of all curves over totally imaginary number fields 
would lead to finiteness of the $u$-invariant of such fields (\cite{Lb-P-S}). 
Finiteness of the $u$-invariant of $k(t)$, $k$ totally imaginary number field, is an open question. 

Let $k$ be a totally imaginary field with   the ring of 
 integers $\OO$. Let  $C/k$  be a smooth projective geometrically integral
 curve over $k$ with  $\CC/\OO$  a regular proper model. 
 In (\cite{Lb-P-S}), finiteness of $\Br_p$dim$(k(C))$ for all $C/k$ is reduced to bounding 
 indices of $p$-torsion elements in $\Br(\CC)$,  i.e. `unramified elements' in $\Br(k(C))$, for all $C/k$. 
By a theorem of Grothendieck, for a smooth projective curve $C$ over $k$,  the image of  $\Br(\CC)$ is  zero in $\Br(k_\nu(C))$  for every finite  place 
  $\nu$ of $k$.  Let  $\Sha \Br(k(C)) = 
  Ker(H^2(k(C),  {\mathbb G}_m) \to \prod_{\nu \in \Omega_K} H^2(k_\nu(C), {\mathbb G}_m))$.
  Note that $\Sha \Br(k(C)) \subseteq \Br(\CC) \subseteq \Br(C)$ (\cite{P-Sj}) and $\Sha \Br(k(C)) = \Br(\CC)$ if $k$
  is a  totally imaginary number field. We state the following conjecture concerning the period/index  bounds for  elements in   $\Sha \Br(k(C))$:
  
 \begin{conj}
For elements in $\Sha \Br(k(C))$,   the index and period coincide.
 \end{conj}
  
If $C$ is an elliptic curve,   $\Sha \Br(k(C))$ is isomorphic to the Tate-Shafarevich group of $C$  and the conjecture is true (\cite{Ca}, \cite{ON}).  
 We  prove this conjecture for  the $2$-torsion subgroup for  genus 2  curves.
 
  \begin{theorem}(cf. \ref{per-ind})
 Let $k$ be a number field and $C$ a smooth projective geometrically integral curve over $k$  of genus 2 with a $k$-rational point. 
 Then for every element $\tilde{\alpha}$ in $ _2\Sha \Br(C)$, index of 
$\tilde{\alpha}$ divides 2.
 \end{theorem}
 
 As a consequence of our approach to the period-index question for hyperelliptic curves, we also obtain  following result of independent interest.

 \begin{theorem}(cf. \ref{lgp})
Let $k$ be a global field of characteristic not $2$. Let $X$ be a smooth intersection of two quadrics 
$Q_1$ and $Q_2$ in $\mathbb{P}^5$.  Suppose $disc(Q_1) = 1$. If $X(k_\nu)$ contains a line for all places $\nu$ of $k$, 
then, $X(k)$ is nonempty.
\end{theorem} 
  
 The Hasse principle for smooth intersection of quadrics in $\P_k^5$ 
 is an open question.  
   Wittenberg (\cite{OW}) proves that assuming  the finiteness of Tate-Shafarevich groups
and Schinzel's hypothesis,  the Hasse  principle holds for smooth intersections of two quadrics in $P^5_{k}$.

 We follow the approach of Lieblich   to study twisted moduli spaces associated to a Brauer 
  class in order  to bound its index. Using this technique, Lieblich proves period index bounds for 
  unramified  Brauer classes of function fields of surfaces over finite fields ((\cite{Lb})). Here is a sketch of the contents of the paper.
  
 In \S \ref{twisted-moduli}, we recall some properties of twisted sheaves and moduli spaces/stacks. 
 In \S\ref{hecke}, we recall the Hecke correspondence, introduced by Narasimhan and Ramanan (\cite{NR}). 
 We describe the correspondence between the moduli spaces of vector bundles of  rank two with odd and
  even degree determinants for curves of genus 2. In \S\ref{twisted-hecke} 
 we define a twist of this correspondence with respect to a Galois cocycle with values in
 $\Pictwo C$.   In \S \ref{conic-fibrations}, we record  a Hasse principle result, due to Colliot-Th\'el\`ene,  for  certain conic fibrations over ${\mathbb P}^3$, which is used in 
 \S \ref{index2-genus2}. 
  In \S \ref{index2-genus2},   if $g = 2$, using  the  twisted  Hecke correspondence, we obtain  period=index statements for 
 $_2\Sha \Br(k(C))$ (cf. \ref{per-ind}) stated at the beginning of the introduction. In particular, if $k$ is a totally 
 imaginary number field, then period and index coincide for elements in $\Br(\CC)$.

  In \S\ref{pencils-hypell}, over a general field of characteristic not $2$,  we recall several results  on a nonsingular pencil  
  $Q=tQ_1-Q_2$ of quadrics on a vector space of dimension $2g+2$ and the asssociated hyperelliptic curve $C$
  defined by $y^2=f(t)$ where $f(t)=disc(tQ_1-Q_2)$. We describe the group of automorphisms $\AAut^+_Q$  of the pencil and  an isomorphism
 $ \theta_Q:  \Pictwo C \simeq \AAut^+_Q$. In \S\ref{twisting-pencils}, given a Galois cocycle $\gamma$ with values
  in $\AAut^+_Q$, we describe the twisted pencil $Q^{\gamma}$ which is defined over a central simple algebra of degree
  $2g+2$. If the algebra is split, we recover the usual pencil of quadrics.
 Let $C$ be  a  smooth hyperelliptic curve  which admits a $k$- rational point $P$ and $L= {\OO}((2g+1)P)$. 
  Let  $M_C(2,L)$ denote the moduli space of rank 2 stable locally free 
  sheaves on $C$ with determinant $L$.   There is a description  due to Desale-Ramanan (\cite{De-Ra}) of $M_C(2, L)$ over $\overline{k}$  in terms of the space 
  $I_{g-1}(Q_0)$, the $(g-1)$-dimensional Grassmannian 
 of linear subspaces contained in the base locus of a pencil of quadrics $ Q_0  =t Q_1 -Q_2$ in 
  $\P^{2g+1}$; the hyperelliptic curve $C$ is defined by 
  $y^2 = (-1)^{g+1}det(Q_0)$. 
  In \S\ref{twisting-moduli-hyperell}, we recall a construction of Ramanan  (cf. \ref{Appendix}) 
  of a   nonsingular pencil $Q_{\eta}$  whose associated hyperelliptic curve is $C$, as well as the isomorphism over $k$:
   $$\psi : M_C(2, L) \to I_{g-1}(Q_{\eta}).$$
 We describe a twist of this isomorphism by a Galois cocycle $\beta$ with values in $\Pictwo C$:
 $$\psi^\beta : M^\beta_C(2, L) \to I_{g-1}(Q^\beta_{\eta})$$ 
Here  $Q_{\eta}^\beta$ is a pencil of twisted quadrics.

 Let  $k$ be a  number field  and $C/k$  a smooth projective 
  hyperelliptic curve of genus $ g$ with $C(k) \neq \emptyset$. 
   Let $\tilde{\alpha} \in ~
   _2\Sha \Br(k(C))$. Then $\tilde{\alpha}$ is in the image of $\Br(C) \to \Br(k(C))$ (\cite{P-Sj}). We choose 
   a lift $\alpha \in H^2(C, \mu_2) $ of $\tilde{\alpha}$ with $\alpha_{\overline{k}} = 0$, where 
   $ \overline{k}$ is the algebraic  closure of $k$.  This is 
   possible since $C(k) \neq \emptyset$. Let $\beta$ be a one cocycle with values in $ \Pictwo C$ associated to 
   $\alpha$ (cf. \S\ref{twisted-moduli},  $(\star \star) $).  In \S\ref{soluble}, we prove (cf. \ref{mc-ig1}) that the 
   twisted pencil $Q_\eta^\beta$ is indeed a pencil of quadrics. This uses a result of Shankar-Wang (\cite{SW}).
Thus we have an identification of the twisted moduli space $M_C^\beta(2, L)$ with the space
  $I_{g-1}(Q_\eta^\beta)$  with $Q_\eta^\beta$ a genuine  pencil of quadrics.   

  In \S\ref{pencil-p5}, we use the 
 identification of the twisted moduli space  $M^\beta_C(2, L)$  with the base locus $X$ of the   pencil  
 $(Q^\beta_{\eta})$ in the genus 2 case and the results of  \S\ref{index2-genus2}  to  derive a variant of a Hasse principle for a smooth intersection of two quadrics in 
${ \mathbb P}^5$  stated in the beginning.

 \vskip 3mm

\noindent
{\bf Acknowledgments:}\\
The first  author was partially supported by the IMSc-DAE project Complex Analysis and Algebraic Geometry.\\
The second author was partially supported by the NSF grant DMS-1801951.\\
The authors thank J.-L. Colliot-Th\'el\`ene,  Max Lieblich, V. Suresh and  J.-P. Tignol for their generous contributions to this paper.

\section{Twisted sheaves and moduli spaces/stacks}
\label{twisted-moduli}

We recall some properties of moduli stacks of twisted sheaves. 
We refer to  the papers \cite{Lb},\cite{Lb2}, \cite{Lb3} of Lieblich 
 for details.

 \subsection{$\mu_n$-gerbes}
 \label{gerbes}
Let $n$ be coprime to char$(k)$ and $C/k$ a smooth projective geometrically integral curve.

\begin{defin}
A $\mu_n$-gerbe  on $C$ is a gerbe $\sheaf{C}\to C$  together with an  isomorphism $(\mu_n)_* \simeq \sheaf{I}(\sheaf{C})$, where
$\sheaf{I}(\sheaf{C})$ denotes the inertial stack of $\sheaf{C}$. 
\end{defin}

We record the following results on $\mu_n$-gerbs on $C$, due to Giraud (\cite{Giraud}, Section IV.3).

\begin{prop}\label{giraud}
There is a natural bijection between the set of isomorphism classes of $\mu_n$-gerbes  on $C$ 
and $H^2(C,\mu_n)$.
\end{prop}
 
 We simply call $\alpha \in H^2(C, \mu_n)$ a $\mu_n$-gerbe under the above correspondence. 
 
 \begin{defin}
 A $\mu_n$-gerbe $\alpha$ is {\it essentially trivial} if its image in $H^2(C, \gG_m)$ is zero. 
 \end{defin}

\begin{defin}
 Let $L$ be an invertible sheaf on $C$. The gerbe $[L]^{\frac{1}{n}}$ is the stack over $C$ whose objects over a $C$- scheme $T$ are pairs 
 $(\LL, \phi)$ where $\LL$ is an invertible sheaf on $T$ and $\phi : \LL^n \to L_T$ is an isomorphism. The gerbe 
 $[L]^{\frac{1}{n}}$ is a $\mu_n$-gerbe. 
\end{defin}

We recall the following facts on essentially trivial $\mu_n$-gerbs from (\cite[\S 2.3.4]{Lb2}). 
The Kummer sequence 
$$ 1 \to \mu_n \to \gG_m \buildrel{n}\over{\to } \gG_m \to 1 $$
yields an exact sequence
$$\Pic C  \buildrel{n}\over{\to } \Pic C \to H^2(C, \mu_n) \to H^2(C, \gG_m).$$
Every essentially trivial $\mu_n$-gerbe $\alpha$ is isomorphic to     $[L]^{\frac{1}{n}}$
where  $L$ is an  invertible line bundle  on $C$ with $\partial(L) = \alpha$. 
There is an invertible $\alpha$-twisted sheaf $\LL_\alpha$ such that  $\LL_\alpha^n \simeq L$,
uniquely determined up to tensoring with a line bundle in $\twoPic C$. 
Given an $\alpha$-twisted sheaf $\FF$, $\FF \otimes \LL_\alpha^{-1}$ is isomorphic to the pull-back to the gerbe $\alpha$
of a sheaf $\FF$ on $C$. This association can be used to define stability of $\alpha$-twisted sheaves.

\subsection{ Twisted sheaves and Stability }

We recall the definition of stability and semistability for locally free sheaves on $C$.
For a vector bundle $E$ on $C$, the slope $\mu(E)$ is the rational number  
$\frac{deg(E)}{rank(E)}$..
The vector bundle $E$ is semistable if $\mu(F) \leq \mu(E)$ for every locally free subbundle $F$ 
of $E$ and stable if strict inequality holds. We shall now   define the corresponding notions for twisted sheaves.

Let $\alpha \in H^2(C, \mu_n)$ and $\FF$ a locally free $\alpha$-twisted sheaf of rank $n$.
Then   determinant  of $ \FF$, denoted by det$(\FF)$,  has no twist and is a pullback of a line bundle $L$ on $C$. 
We write det$(\FF) = L$. 

For an invertible  $\alpha$-twisted sheaf $M$, $M^{\otimes n} = L$ where $L$ is a line bundle on $C$. 
We define deg$(M) = \frac{1}{n} deg(L)$. We define 
the degree of a locally free $\alpha$-twisted sheaf $\FF$  to be the degree of det$(\FF)$.

\begin{defin}\label{mustable}
The  slope $\mu(\FF)$  of a locally free  $\alpha$-twisted 
sheaf $\FF$  is  $\frac{deg(\FF)}{rank(\FF)}$.
\end{defin}

An $\alpha$-twisted locally free sheaf $\FF$ is {\it semistable} (resp. {\it stable}) if for every 
$\alpha$-twisted subbundle $\sheaf{G} $ of $\FF$, 
$$
\mu(\sheaf{G})\leq \mu(\sheaf{F})\,\,\, ( {resp. }\mu(\sheaf{G})< \mu(\sheaf{F})). 
$$

Let $\alpha \in H^2(C, \mu_2)$. Since $H^2(C_{k^s}, \gG_m) = 0$, $\alpha_{k^s}$ is essentially trivial and there is a
line bundle $L$ on $C_{k^s}$ such that the image of 
$[L] \in (\Pic C_{k^s}/2)  \to H^2(C_{k^s}, \mu_2)$ is $\alpha$.  Thus there is an $\alpha$-twisted invertible sheaf $\LL_\alpha$
over $k^s$. 
The stability /semistability of an $\alpha$-twisted locally free sheaf $\FF$ is equivalent to the stability/semistability 
of the locally free sheaf  $\FF \otimes \LL_\alpha^{-1}$ over $C_{k^s}$.

\subsection{Moduli stacks}
\label{moduli-stacks}

Let $\alpha \in H^2(C, \mu_2)$ and $L$ a line bundle on $C$. 
The moduli stack $\sheaf{M}_C^\alpha(n,  L)$ of rank $n$  $\alpha$-twisted stable locally free sheaves 
  determinant $L$ on $C$  is defined in \cite{Lb}.
It was shown to be an Artin stack  locally of finite presentation over $k$ (see \cite{Lb}).
In fact,  the stack $\sheaf{M}_C^\alpha(n, L)$ is a $\mu_n$-gerbe over the  smooth quasi-projective variety 
$M_C^\alpha(n, L)$ 
which is the   coarse moduli space of the stack  $\sheaf{M}_C^\alpha(n,  L)$. 
This variety  admits a locally factorial compactification $M^{\alpha, ss}_C(n, L)$  with a closed  
complement of  codimension at least one. 
In fact $M^{\alpha, ss}_C(n, L)$ is the coarse moduli space of stack of rank $n$ locally
 free semistable $\alpha$-twisted sheaves of determinant 
$L$. 

We record the following result from (\cite[3.1.2.1]{Lb}). 

\begin{theorem}
Let $\alpha \in H^2(C,\mu_n)$ with image $\tilde{\alpha} \in H^2(C,\gG_m)$. 
Then index$(\tilde{\alpha})$ divides $n$ if and only if there is a stable locally free $\alpha$-twisted 
sheaf of rank $n$ and trivial determinant. 
\end{theorem} 

There is another way of describing the stack of $\alpha$-twisted sheaves and the associated 
moduli space $M_C^\alpha(n,L)$ using  Galois twists which we shall explain next. Let $f : C \to k$ be a 
smooth projective geometrically integral curve with $C(k) \neq \emptyset$. 
The Leray spectral sequence yields decomposition 

\hspace*{\fill}
$H^2(C, \gG_m) = H^2(k, \gG_m) \oplus H^1(k, \Pic C)$    \hfill  $(\star)   $
\hspace*{\fill}
 
\hspace*{\fill}
$  H^2(C,  \mu_n) = H^2(k, \mu_n ) \oplus H^1(k, \nPic C) \oplus R^2f_*\mu_n $  \hfill $(\star \star)  $
\hspace*{\fill}

\noindent
The sheaf $R^2f_*\mu_n$ is isomorphic to $\Z/n\Z$. The splitting of the 
natural map $H^2(C, \mu_n) \to \Z/n\Z$ is given by $[1]$ mapping to the gerbe 
$[\OO(P)]^{\frac{1}{n}}$. 

Let $\alpha \in H^2(C,\mu_n)$ with image $\tilde{\beta} \in H^1(k, \nPic C)$ under projection in $(\star \star)$. 
Let $\beta : \Gamma_k \to  \nPic C$  be a cocycle representing $\tilde{\beta}$. 
Let $\sheaf{M}_C(n, L)$ (resp. $\sheaf{M}_C^{ss}(n , L)$) denote 
the stack of rank $n$ stable (resp. semistable) locally free sheaves of $C$ with determinant $L$. 
Tensoring gives an action of $\nPic C$ on these moduli stacks and  we have the twists 
$\sheaf{M}_C^\beta(n, L)$ (resp. $\sheaf{M}_C^{ss, \beta}(n, L)$) of these moduli stacks by the cocycle $\beta$.
We also have the twisted $M_C^\beta(n, L)$  and $M_C^{ss, \beta}(n, L)$. 
In view of (\cite[3.2.2.6]{Lb}),  the moduli space $M_C^\alpha(n, L)$  of $\alpha$-twisted stable sheaves of rank $n$
and determinant $L$ is precisely  the twist $M_C^\beta(n, L)$ of the moduli space $M_C(n, L)$ by the cocycle $\beta$. 
We shall identify the space $M_C^\alpha(n, L) $  with   the 
corresponding twisted moduli spaces.

 \section{Hecke correspondence }
 \label{hecke}
 
 In this section, we recall the Hecke correspondence, introduced by
  Narasimhan and Ramanan (\cite{NR}]). This was further investigated by 
 Beauville-Laszlo-Sorger (\cite{Beauville}). We shall use the terminology and notations from  (\cite[p. 206]{Beauville}). 
 We recall the definition of the stability index $s(E)$ of a vector bundle $E$
 on $C$. 
 \begin{defin}
 The stability index $s(E)$  is defined to be the minimum of
$\mu(E'') - \mu(E')$ as $(E', E'') $  vary over pairs of locally free sheaves on $C$ which fit into an exact sequence 
$$0 \to E' \to E \to E'' \to 0.$$
\end{defin}
\noindent
The  bundle $E$ is stable if and only  if $s(E) >0$ and semistable if and only if $s(E) \geq 0$.\\

Let $C \to {\mathbb P}^1$ be a hyperelliptic curve of genus $g \geq 2$ 
defined over a field $k$ of characteristic not 2. We assume that 
$C(k) \neq \emptyset$.
For any line bundle $L$,  let   $M_C(r, L)$ be the moduli space of stable  rank $r$
 locally free sheaves of determinant 
$L$ and $M^{ss}_C(r, L)$ the moduli space of semistable  rank $r$
 locally free sheaves of determinant 
$L$.  Let $P$ be a closed point of $C(k)$. For $r \geq 2$, let $$ \EE \to C \times M_C(r, \OO_C(P))$$
be the Poincar\'e bundle (\cite{Tj}, \cite{Ra}).  Let $\EE_P$ be the restriction of $\EE$ to $ P \times M_C(r, \OO_C(P))$.
For $0 < h < r$, let $\PP = Gr(h, \EE_P)$ be the Grassmannian bundle parametrizing rank $h$
 locally free quotients of $\EE_P$.  A closed point of $\PP$ is a pair $(E, F)$ of vector bundles on $C$ with 
$E \in M_C(r, \OO_C(P))$ and $E(-P) \subset F \subset E$ and dim$(E_P/F_P) = h$ (\cite[Section 10]{Beauville}).

We have a morphism $s: \PP \to M_C(r, \OO_C(P)$ given by $s(E, F) = E$. This map is a 
Grassmann bundle. We recall the following Lemma from (\cite[Lemma 10.2]{Beauville}).

\begin{lemma}
\label{semistable} 
Let $(E, F) \in \PP$. If $g \geq 2$ and $E$ is general, i.e., $s(E) = g-1$, then $F$ is semistable with 
determinant $ \OO_C((1-h)P)$. 
\end{lemma}

We thus have the Hecke correspondence 

$$
\begin{array}{cccc}
\PP &  \buildrel{s'}\over{ - ~ -  \longrightarrow} & M_C^{ss}(r,  \OO_C((1-h)P)) \\
s\Big\downarrow & &  & \hskip 2cm  (\star) \\
M_C(r, \OO_C(P)).
\end{array}
$$
where the rational map $s'$ is defined by $s'(E,F) = F$ if $E$ is general.  Let  V be the open subscheme of 
$M_C(r, \OO_C(P)$ consisting of $E$ with $s(E)=g-1$ and $U = s^{-1}(V)$. 
Then $U$ is an open subscheme of $\PP$ and $s'/U: U \to M_C^{ss}(r,  \OO_C((1-h)P))$ is a morphism.

In the above Hecke correspondence we may replace $\OO_C(P)$ by any line bundle L of degree 1 and get a  Hecke correspondence

$$
\begin{array}{cccc}
\PP &  \buildrel{s'}\over{ - ~ -  \longrightarrow} & M_C^{ss}(r,  L\otimes \OO_C((-h)P)) \\
s\Big\downarrow & &  & \hskip 2cm  (\star) \\
M_C(r, L).
\end{array}
$$
\noindent
By tensoring with a suitable power of $\OO_C(P)$, one also gets a Hecke correspondence

$$
\begin{array}{cccc}
\PP &  \buildrel{s'}\over{  - ~ -   \longrightarrow} & M_C^{ss}(r,  L'\otimes \OO_C((-h)P)) \\
s\Big\downarrow & &  & \hskip 2cm  (\star) \\
M_C(r, L').
\end{array}
$$
\noindent
for any line bundle $L'$ of odd degree.

\begin{lemma}
\label{slope1}
Let $g=2$ and $L$ a line bundle of degree $1$ on $C$. Then for $E \in M_C(2, L)$, $s(E)=1$.
\end{lemma}

\begin{proof}
Since the stability index of a rank 2 vector bundle is an integer, $E$ being stable, $S(E) \geq 1$.

Let 
$$0 \to L' \to E \to L'' \to 0$$
be an exact sequence with $L', L"$ line bundles  on $C$. Then,
$s(E) \geq  \mu(L'') - \mu(L')  = deg(L'') - deg(L') = deg(E) - 2deg(L')$. 
Since $E$ is stable, for any subbundle $L'$ of $E$, $\mu(L') < \mu(E) = 1/2$ so that  $deg(L')
\leq 0$. 
We now prove that there  is a line subbundle $L_0$ of $E$ of 
degree zero. This will prove that $s(E) =1$.

In view of (\cite[Lemma 10.2]{NR}), there is a   line  bundle $L_0$ on $C$  of degree 0
and an injective map $L_0 \to E$.  Indeed $L_0$ is a line subbundle of $E$; otherwise it would 
generate a line subbundle $\tilde{L}_0$ of degree $\geq 1$, contradicting $\mu(E) = \frac{1}{2}$. 
Thus $s(E) = 1$. 
 \end{proof}
 
 For the rest of the section, we assume that genus of $C$ is 2, $r = 2$.  In view of (\ref{slope1}), for a line bundle $L$ of degree 1,   
 every element of $M_C(2, L)$ has stability index $1$ so that the morphism $s'$ is defined on the whole of
 $\PP$ and we have the Hecke diagram
 
 $$
\begin{array}{cccc}
\PP &  \buildrel{s'}\over{ \longrightarrow} & M_C^{ss}(2,  L\otimes \OO_C(-P)) \\
s\Big\downarrow & &  & \hskip 2cm  (\star) \\
M_C(2, L).
\end{array}
$$
where $s$ and $s'$ are both morphisms and $s$ is a ${\mathbb P}^1$ bundle.
The map $s'$ is a conic bundle.  The following lemma asserts that each fiber of $s'$
 at a point representing a stable bundle in $M_C^{ss}(2,  L\otimes \OO_C(-P))$
 is isomorphic to ${\mathbb P}^1$.

\begin{lemma}
\label{p1fibres} 
Let  $x  \in M_C^{s}(2,  L\otimes \OO_C(-P))(k)$ be a  stable point represented by rank 2 vector bundle $F$ on $C$.
 Then, $x$ belongs to the image of $s'$ and the fiber of $s'$ at $x$ is isomorphic to ${\mathbb P}^1$.
\end{lemma}

\begin{proof}
Let $E'$ be a rank two vector bundle satisfying $F\otimes \OO_C(-P) \subset E' \subset F$ and $dim(F_P/E'_P) = 1$.
Then $det(E')  = det(F)\otimes \OO_C(-P)) = L\otimes \OO_C(-2P)$ 
 so that $deg(E') = -1$.  We claim that 
$E'$ is stable. Suppose not. Then there is a line subbundle $L'$ of $E'$ 
such that $deg(L') \geq \mu(E') =-1/2$. Thus $deg(L') \geq 0$. Let ${\tilde L'} $
 be the line subbundle of $F$ spanned by $L'$. Then 
$deg({\tilde L'}) \geq deg(L' )\geq 0 $. However $F$ is stable and $deg({\tilde L'}) <\mu(F)=0$. 
This leads to a contradiction. Thus $E'$ is stable. We have $ E' \subset F \subset E'\otimes \OO_C(P)$ with 
$E'\otimes \OO_C(P)$ stable and $det(E'\otimes \OO_C(P)) = L $. 
The pair $(E' \otimes \OO_C(P), F) \in \PP$ and maps to $F$ under $s'$. The bundles $E'$ satisfying 
$F\otimes \OO_C(-P) \subset E' \subset F$  and $dim(F_P/E'_P) = 1$ 
are pametrized by the 1-dimensional quotients of the vector space $F_P$. 
Thus the fiber at $x$ of the map $s'$ is isomorphic to $Gr(1, F_P) \simeq  \mathbb P^1$. 
\end{proof}

 \section{Twisted Hecke Correspondence }
 \label{twisted-hecke}

Let  $C$ be a smooth projective  hyperelliptic curve over a field $k$ with a rational point $P$
. Let $L$ be a line bundle of degree 1 over $C$. Let 

$$
\begin{array}{cccc}
\PP &  \buildrel{s'}\over{ - ~ -  \longrightarrow} & M_C^{ss}(r, L\otimes \OO_C((-h)P)) \\
s\Big\downarrow & &  & \hskip 2cm (\star) \\
M_C(r, L).
\end{array}
$$
\noindent
be the Hecke Correspondence as in  \S\ref{hecke}. Tensoring gives an action of $\rPic C$ on $M_C(r, L)$ 
 and  $M_C(r,{\OO}_C(-hP))$. There is also an action of $\rPic C$ rank  r  sheaves  $(E, F)$ on $C$ with 
$E \in M_C(r,\OO_C(P)),  E(-P) \subset F \subset E$ and dim($E_P/F_P) = h$. 
The maps $s, s'$ are compatible with the action of $_rPic(C)$.

Let $\tilde{\alpha} \in H^2(C, G_m)$ be an $r$-torsion element with a lift $\alpha \in H^2(C, \mu_r)$
such that $\alpha_{k^s} = 0$. Such a lift is possible since $C(k)$ is not empty. 
 Let $\beta: \Gamma_k \to~  _rPicC$ be a cocycle whose class $\tilde{\beta} \in H^1(k, _rPicC) $
  is the image of $\alpha$ under the projection in (**). Then $\beta$ gives twists $\PP^{\beta}$, $M_C(r, L)^{\beta}$ and 
$M_C(r, L \otimes\OO_C(-hP))^{\beta}$ together with morphisms

$$
\begin{array}{cccc}
\PP^\beta &  \buildrel{s'}\over{ - ~ -   \longrightarrow} & M_C^{ss. \beta}(r,  L\otimes\OO_C(-h)P))^\beta  \\
s\Big\downarrow & & & \hskip 2cm (\star\star) \\
M_C^\beta(r, L).
\end{array}
$$

We call this diagram a twisted Hecke correspondence. We may think of this as a Hecke correspondence 
between twisted sheaves of odd and even determinant in the case $r=2$ and $h=1$.

\begin{remark} Suppose $r=2, h=1$ and $\alpha \in H^2(C, \mu_2)$  is essentially trivial, i.e., 
$\tilde{\alpha} = 0$. There is a line bundle $N$ over $C$ such that under 
$\delta: PicC/2 \to H^1(k, \Pictwo C)$, the class of $N$ maps to $\tilde{\beta}$.
 We note that $\alpha_{k^s} =0$ implies that the degree of $N$ is even.
  Let ${\mathcal N}_{\alpha}$ be an invertible $\alpha$-twisted sheaf such that 
$({\mathcal N}_{\alpha})^2 = N$ (cf. \S\ref{twisted-moduli}).  
The map $(M_C^{\beta}(2, \OO_C(P))\to M_C(2, \OO_C(P))\otimes N)$ and 
 the map $M_C^{ss, \beta}(2, \OO) \to M_C^{ss}(2, N)$
given by tensoring with ${\mathcal N}_{\alpha}$ are isomorphisms
 which yield untwisting of the twisted moduli spaces. The twisted Hecke correspondence diagram

$$
\begin{array}{cccc}
\PP^\beta &  \buildrel{s'}\over{ - ~ -   \longrightarrow} & M_C^{ss. \beta}(2,  \OO))  \\
s\Big\downarrow & & &\hskip 2cm  (\star\star) \\
M_C^\beta(r, \OO_C(P)).
\end{array}
$$

reduces to the usual Hecke correspondence

$$
\begin{array}{cccc}
\PP  &  \buildrel{s'}\over{ - ~ -   \longrightarrow} & M_C^{ss}(2,  N)   \\
s\Big\downarrow & & & \hskip 2cm (\star\star) \\
M_C(2, N\otimes \OO_C(P)) 
\end{array}
$$
with degree of $N \otimes \OO_C(P)$ odd. 
\end{remark}

 \section{Local global principle for conic fibrations}
 \label{conic-fibrations}
 In this section we give a proof of a local global principle for rational points on certain conic
fibrations over global field of characteristic not 2 due to Colliot-Th\'el\`ene.  We begin with the following result of Suresh.
   
   \begin{theorem}(Suresh) 
   \label{suresh}
   Let $k$ be a field of characteristic zero and $K = k(t_1, \cdots , t_n)$.
Let $\alpha$ be a quaternion algebra over $K$ which is unramified on ${\mathbb P}^n_k$ except possibly at a unique 
codimension one  point  of ${\mathbb P}^n_k$.  Let $C$ be the conic over $K$ associated to $\alpha$.
Then $\Br(k) \to \Br_{nr}(K(C)/k)$ is onto. 
\end{theorem}
 
\begin{proof}
Let $\tilde{\beta} \in  \Br_{nr}(K(C)/k)$. Since $C$ is a conic and $ \Br_{nr}(K(C)/k) \subseteq \Br_{nr}(K(C)/K)$,
$\tilde{\beta }$ is in the image of $\Br(K)$ (\cite[Proposition 6.2.1]{CTSk}). Let $\beta \in \Br(K)$ which maps to $\tilde{\beta}$.

Let $Z$ be a codimension one point of ${\mathbb P}^n_k$ and $\nu$ the discrete valuation of $K$ given by $Z$. 
Let $\kappa(\nu)$ be the residue at $\nu$.
Let $\ell/k(\nu)$ be the cyclic extension given by the residue of $\alpha$ at $\nu$. 
Then there exists a discrete valuation $w$ on $K(C)$ extending $\nu$ with $\ell$ as the algebraic closure of 
$\kappa(\nu)$ in  the residue field $\kappa(w)$ at $w$ (\cite[Lemma 6.1]{PSIMRN}).  
We have $\partial_\nu(\beta)_{ \kappa(w)} = \partial_w(\tilde{\beta}) = 0$. 
Further  if $\ell \neq \kappa(\nu)$, then $\ell$  is the unique quadratic extension of  $\kappa(\nu)$   contained in $\kappa(w)$.

Suppose that $Z \neq Y$. Then $\ell = \kappa(\nu)$.  Since $\kappa(\nu)$ is algebraically closed in $\kappa(w)$, 
$H^1(\kappa(\nu), {\mathbb Q}/{\mathbb Z}) \to H^1(\kappa(w), {\mathbb Q}/{\mathbb Z}) $
is injective and $\partial_\nu(\beta) = 0$.

Suppose that $Z =  Y$. Suppose $\partial_\nu(\beta) = 0$. Then $\beta$ is unramified on ${\mathbb P}^n$ and hence comes from a class in $\Br(k)$. 
Suppose $\partial_\nu(\beta) \neq 0$. Since $\ell/\kappa(\nu)$ is a degree two extension and $\ell$ is the algebraic closure of $\kappa(\nu)$ 
in $\kappa(w)$,  $\partial_\nu(\beta) = \ell = \partial_\nu(\alpha)$. 
Since $\alpha$ and $\beta$ are unramified except at $Y$, $\beta- \alpha$ is unramified on ${\mathbb P}^n$.
Hence $\beta - \alpha$ comes from $\Br(k)$. Since the image if $\alpha$ in $\Br(K(C))$ is zero, 
$\tilde{\beta}$ is in the image of $\Br(k)$.
\end{proof}

\begin{remark}
If char$(k) = p > 0$ and not equal 2, then the above proof gives a surjection $\Br'(k) \to \Br'_{nr}(K(C)/k)$ where $\Br'$ 
denotes the subgroup of $\Br$  whose torsion is prime to $p$.  
\end{remark}

\begin{remark} 
In fact in the above theorem we can replace ${\mathbb P}^n$ by any smooth projective variety $X$ over $k$
with $\Br(k) \to \Br_{nr}(k(X)/k)$ onto. 
\end{remark}

We recall the following. 

\begin{theorem}\cite[Th\'eor\`eme 2]{CT88} 
\label{colliot}Let $k$ be a number field and  $X/k$ be a smooth surface with $X \to {\mathbb P}^1$ be a conic bundle
with at most 4 geometric  degenerate  fibres. Then  the Brauer-Manin obstruction  is the only obstruction to the Hasse principle for $X$. In particular, if
$\Br(k)  \to \Br(X) $ is onto, then the Hasse principle holds for $X$.
\end{theorem}
  
 This theorem is an extension of the case of Ch\^atelet surfaces (\cite[Theorem 8.11]{CTSaSwD87}). See also   
 Salberger \cite[Cor 7.8]{Sal88}.  
 
 \begin{lemma} 
 \label{hilbert}
 Let $Y \subset {\mathbb P}^{n+1}$ be an  integral hypersurface of degree $d$ over a  number  field $k$. 
 Let $A$ be a $k$-point  in ${\mathbb P}^{n+1}$
 outside $Y$. Let $F_A$ be the Grassmannian of lines in ${\mathbb P}^{n+1}$ passing through $A$. Then there is a Hilbertian set $H \subset F_A$ of 
 $k$-rational points such that for each $\ell \in H$, $\ell \cap Y$ is integral.   
 \end{lemma}
 
 \begin{proof}
 We fix a hypersurface $T \subset {\mathbb P}^{n+1}$ with $A \not\in T$. The morphism ${\mathbb P}^{n+1} \setminus \{ A \}\to T$ sending a point $B$ to the intersection of the 
 line $AB$ joining $A$ and $B$  with $T$ restricts to a morphism $f : Y \to T$ of degree $d$. 
 For $M \in H$, the fibre $f^{-1}(M)$ is simply the line $AM$  intersected with $Y$. Since $T \simeq {\mathbb P}^n $ and 
 $k(Y) /k(T)$ is a finite extension of degree $d$, by Hilbert irreducibility theorem, there is a Hilbertian set $\tilde{H}$ of $k$-rational points in $T$ such that for every 
 $M \in \tilde{H}$, the fibre $f^{-1}(M)$ is integral. We have an isomorphism $T\buildrel{\psi}\over{\to} F_A$ sending a point $M \in H$ to the line $AM$. 
 The set  $H =  \psi(\tilde{H}) \subset F_A$ is the required Hilbertian set.  
 \end{proof}
 
 \begin{theorem} (Colliot-Th\'el\`ene)
 \label{colliot-hasse}
 Let $k$ be a number  field. Let $ n \geq 2 $ be an integer.  Let $Y \subset {\mathbb P}^n_k$ be a 
  geometrically integral quartic hypersurface. Let $U \subset {\mathbb P}^n_k$ be the complement of $Y$. Let $f : V \to U$ be
 a smooth conic bundle.   Then
the Hasse principle  holds for $V$.
  \end{theorem}
  
  \begin{proof} Suppose $\prod_\nu V(k_\nu) \neq \emptyset$. 
  
 Let $A \in U(k)$ and  $F_A$ be the Grassmannian of lines in ${\mathbb P}^{n}$ passing through $A$. 
 Then,  by (\ref{hilbert}),   there is a Hilbertian set $H \subset F_A$ of 
 $k$-rational points such that for each $\ell \in H$, $\ell \cap Y$ is a closed point of degree 4.    
 
 Let  $V_A = f^{-1}(A)$ be the fibre of $f$ at $A$.  Then $V_A$ is a smooth conic and hence locally has rational points at all places outside a finite set 
$S$ of places of $k$. For each $\nu \in S$, we pick a point $M_\nu \in V(k_\nu)$ and set $B_\nu = f(M_\nu) \in U(k_\nu)$. Since 
$f : V \to U$ is smooth, by implicit function theorem, there is a $\nu$-adic neighborhood $\Omega_\nu$ of $B_\nu$ and a local section of $f : V(k_\nu) 
\to U(k_\nu)$ above $\Omega_\nu$. We consider for each $\nu \in S$, the line $\ell_\nu \in F_A(k_\nu)$ through $A$ and $B_\nu$ in ${\mathbb P}^n_{k_\nu}$. By 
a theorem of Ekedahl (\cite{Ek89}), since $Y$ is geometrically integral, $F_A$ is  rational and $H \subset F_A(k)$ is Hilbertian, the image of $H$ in $\prod_{\nu \in S} F_A(k_\nu)$ is dense.
We pick a line $L \in H$ such that the image of $L$ in $F_A(k_\nu)$ is close enough to $\ell_\nu$ such that $L(k_\nu) \cap \Omega_\nu \neq \emptyset$ for all $\nu \in S$. 
We then have the following properties for $L$ :-\\
1) $ \tilde{f} : V_L \to L$ is a smooth conic fibration except at a closed point of degree 4; $\tilde{f}$ is the restriction $f$ to $V_L$. \\
2) $A \in L(k)$ and the conic $\tilde{f}^{-1}(A) = V_A$ has a rational point  over $k_\nu$ for all $\nu \not\in S$. In particular,  $V_L(k_\nu) \neq \emptyset$ for all $\nu \not\in S$\\
3) For $\nu \in S$, $L(k_\nu) \cap \Omega_\nu \neq \emptyset$ so that the section to $f$ over $\Omega_\nu$ yields a $k_\nu$-rational point on $V_L$; in other words 
$V_L(k_\nu) \neq \emptyset$ for all $\nu \in S$.  \\
4) $L \cap Y$ is a closed point of degree 4 since $L \in H$.

Thus the conic fibration $V_L \to L$ has the property: $\prod_\nu V_L(k_\nu) \neq \emptyset$ and it is smooth outside a closed point of degree 4 in $L$.
 By (\ref{suresh}, \ref{colliot}), $V_L(k) \neq \emptyset$. In particular $V(k) \neq \emptyset$.  
  \end{proof}

\section{Period-index questions for genus 2 curves over number fields }
\label{index2-genus2}

Let $C$ be a smooth projective geometrically integral curve over a global field $k$ of characteristic not dividing $n$. 
We begin with the following proposition relating the existence of rational points on the  twisted moduli stack to
the existence of rational points on  the twisted moduli space associated to a $\mu_n$ gerbe on C.
 The proof of the proposition is due to Lieblich.

\begin{prop} (Lieblich)
\label{lieblich}
Let $k$ be a global field of characteristic not $n$ and $C$ a smooth projective geometrically integral curve over $k$.
 Let $L$ be a line bundle of degree one on C. Let $\alpha \in H^2(C, \mu_n)$ be a gerbe which is locally essentially trivial 
 and $\alpha_{k^s} = 0$. Let  $\sheaf{M}_C^\alpha(n,L)$ denote the moduli stack of rank $n$  
 $\alpha$-twisted stable locally free sheaves of determinant $L$ and $M_C^\alpha(n ,L)$ 
 the corresponding coarse moduli space. Then $\sheaf{M}_C^\alpha(n,L)(k) \neq \emptyset$ if and only if 
 $M_C^\alpha(n,L)(k) \neq \emptyset$.
\end{prop} 

\begin{proof}
The map  $\sheaf{M}_C^\alpha(n,L) \to M_C^\alpha(n ,L) $ is a $G_m$ gerbe. Let  $M= M_C^\alpha(n,L)$. 
Suppose $M(k) \neq \emptyset$.  
The variety $M$  is geometrically rational with $Pic M_{\bar k}
\simeq {\mathbb Z}$ and $\Br(M_{\bar k}) = 0$ (\cite{Lb2}).  The spectral sequence

$$H^i(k, H^j(M_{\bar k}, G_m))  \implies  H^{i+j} (M, G_m))$$
yields a surjection  $\Br(k) \to \Br(M)$. Since $M(k) \neq \emptyset$, this map is an isomorphissm.
Thus the gerbe  $\sheaf{M}_C^\alpha(n,L) \to M$ is defined by a constant Brauer class $\beta \in \Br(k)
= \Br(M)$. Since $\alpha$  is locally essentially trivial, $\alpha = \delta ([N_\nu])$ for some line bundle $N_\nu$ on $C$. 
 Here $\delta: PicC/2 \to H^1(k, \Pictwo C) $ is the Kummer connecting homomorphism.
By (\S \ref{gerbes}),  $\sheaf{M}_C^{\alpha_{k_\nu}} (n, L) \simeq \sheaf{M}_C(n, L \otimes N_\nu)$.

Since there exist rank $n$ stable vector bundles with a prescribed determinant,
 $\sheaf{M}_C(n, L \otimes N_\nu)$ has a $k_v$-rational point. Hence the gerbe $\sheaf{M}_C^{\alpha_{k_\nu}}(2, L) \to 
 M_C^{\alpha_{k_\nu}}(n, L)$ is the trivial gerbe  so 
  that the class of $\beta$ in $ \Br(k_{\nu}) $ is zero. Since this is true for 
  all places $\nu$ of $k$, the class $\beta$ is zero in $\Br(k) = \Br(M)$. 
 Thus the gerbe $\sheaf{M}_C^\alpha(n,L) \to M$ is the trivial gerbe. Since $M(k)$ is not empty, 
 $\sheaf{M}_C^\alpha(n,L) (k)$ is not empty. 
\end {proof}

\begin{theorem}
\label{mainthm}
Let $k$ be a number field and $C$ a smooth projective geometrically integral 
hyperelliptic curve over $k$ of genus 2. Suppose $ C(k)$ is not empty. Let $\tilde{\alpha} \in \Brtwo(C) $  be a class which is  
zero in $\Br(C_{k_\nu}) $ for all places $\nu$  of  $k$. Let $\alpha \in H^2(C, \mu_2)$ be a lift of
 $\tilde{\alpha}$
with $\alpha_{\bar{k}} =0$.  Then $M_C^{ss, \alpha} (2, \sheaf{O}_C)$ is isomorphic to ${\mathbb P}^3$.
\end{theorem}

\begin{proof} 
Since the genus of $C$ is two, $M_C^{ss}(2, \sheaf{O}_C)_{\overline{k}}  \simeq {\mathbb P}^3_{\overline{k}} $ (\cite[\S 7, Theorem 2]{NR}).
Since $M_C^{ss, \alpha}(2, \sheaf{O}_C)$ is a twist of  $M_C^{ss}(2, \sheaf{O}_C)$, 
it is a Severi-Brauer variety associated to a central simple algebra $A$ of degree $4$ over $k$. 
The variety admits a rational point if and only if the algebra $A$ is a matrix algebra. For every place $\nu$ of $k$, the element 
$\alpha_{k_\nu}$ is essentially trivial. By  (\S \ref{gerbes}), $(M_C^{ss, \alpha}(2, \sheaf{O}_C)_{k_\nu} \simeq
M_C^{ss}(2, N_\nu)$  where $N_\nu \in PicC_{k_\nu}$ and $\delta ([N_\nu] ) = [\alpha_{k_\nu}]$ in 
$H^2( C_{k_\nu}, \mu_2)$. Hence $M_C^{ss, \alpha}(2, \sheaf{O}_C)(k_\nu) \neq \emptyset$.
This implies that $A \otimes k_\nu$ is the matrix algebra for every place $\nu$ of $k$. 
By Hasse-Brauer-Noether-Albert theoem, the algebra A is a matrix algebra over k. 
Hence  $M_C^{ss, \alpha}(2, \sheaf{O}_C)$ is isomorphic to ${\mathbb P}^3$ 
\end{proof}

\begin{theorem}
\label{essentially-trivial}
Let $k$ be a totally  number field and $C$ a smooth projective geometrically 
integral hyperelliptic curve over $k$ with a $k$-rational point. Let $\alpha \in H^2(C, \mu_2)$
 be a locally essentially trivial element
with $\alpha_{\bar{k}} = 0$.
Then for any $P \in C(k)$, $M_C^{\alpha}(2, \OO_C (P))(k) \neq \emptyset$.
\end{theorem}

\begin{proof}
We have the twisted Hecke correspondence:

$$
\begin{array}{cccc}
\PP^\alpha &  \buildrel{s'}\over{\longrightarrow} & M_C^{ss, \alpha}(2,  \OO_C)  \\
s\Big\downarrow & & & \hskip 2cm (\star\star) \\
M_C^\alpha(2, \OO_C(P)).
\end{array}
$$
By (\ref{mainthm}), $ M_C^{ss, \alpha}(2,  \OO_C) \simeq \mathbb{P}^3$.
Further the stable moduli space $ U = M_C^{s, \alpha}(2,  \OO_C)$ is a open subscheme of
$M_C^{ss, \alpha}(2,  \OO_C)$ whose complement is a geometrically integral  quartic surface $Y$ which over the 
algebraic closure is the Kummer surface associated to the canonical involution $x \mapsto -x$ on $J_C$ (\cite[Section 4, Remark 2]{N1R1}).
Let $V = s'^{-1}(U)$. Then $s' : V \to U$   is a smooth conic fibration.

Let $\nu$ be a place of $k$. Since $\alpha_{k_\nu}$  is  essentially  trivial, 
the Hecke diagram   untwists over $k_\nu$.  Since $C_{k_\nu}$ admits a rank 2 stable bundles 
of any determinant and the untwisted fibre over such points is the projective line (\ref{p1fibres}), 
it follows that $V(k_\nu) \neq \emptyset$. 
Hence by (\ref{colliot-hasse}), $V(k)$ is not empty.  Any rational point of $V$ maps to a rational point of 
 $M_C^{\alpha}(2, \OO_C (P))$.
 \end{proof}

\begin{cor}
\label{per-ind}
Let $k$ be a number field and $C$ a smooth projective geometrically integral hyperelliptic curve over $k$  of genus 2 with a
 zero-cycle of degree one. Then for every element $\tilde{\alpha}$ in $ _2\Sha \Br(C)$, index of 
$\tilde{\alpha}$ divides 2.
\end{cor}

\begin{proof} Since the index of  elements in $_2\Br(C)$ does not change  over an odd degree of $k$, 
we assume that $C(k) \neq \empty$. 
If $\alpha \in H^2(C, \mu_2)$ is a lift of  $\tilde{\alpha}$  with $\alpha_{\bar{k}} =0$, then, by hypothesis,
$\alpha$ is locally essentially trivial. Hence by  (\ref{essentially-trivial}),  $M_C^{\alpha}(2, \OO_C (P))$ 
admits a $k$-rational point. By (\ref{lieblich}),  $\sheaf{M}_C^\alpha(n,L)(k) \neq \emptyset$. If $\sheaf{E}$
is an $\alpha$ twisted sheaf of rank two, $End(\sheaf{E}) $ is a degree two Azumaya algebra representing 
the class of  $\tilde{\alpha}$ in the Brauer group of $C$. Thus  the index of $\tilde{\alpha}$ divides 2.

\end{proof}

\begin{cor}
Let $k$ be a totally imaginary number field and $C$ a smooth projective 
geometrically integral  curve  of genus 2 over $k$  with a  zero-cycle of degree one.
 Let $\mathcal{C}$
 be a regular proper model of $C$ over the ring of integers $\OO $ in $k$. 
 Then for an $\tilde{\alpha} \in \Brtwo(\mathcal{C})$, the  index of $\tilde{\alpha}$ divides 2.
\end{cor}

\begin{proof}
We need only to verify that $\Br\mathcal{C}$ is contained in  $\Sha \Br(C)$.  Over the valuation ring 
${\OO}_\nu$, $\mathcal{C}_{\OO_\nu}$ is a relative curve which is regular and the
 special fiber is a curve over a finite field. By a theorem of Grothendieck (cf. \cite[Thm 10.3.1]{CTSk}), the Brauer group of $\mathcal{C}_{\OO_\nu}$ is zero.
\end{proof}

\section{Pencils of quadrics and hyperelliptic curves} 
\label{pencils-hypell}
Let $k$ be a field of characteristic not 2.
Let $Q_1$ and $Q_2$ be two regular quadratic forms on a vector space $V$ of 
dimension $2n+2$, $n \geq 2$ over $k$.
Let $Q = tQ_1  -  Q_2$ be the associated pencil. Let $f(t) = $ disc$(Q) = (-1)^{n+1}$det$(Q) \in k[t]$. 
The polynomial is determined up to a square in $k^*$. 
We assume throughout that $f(t)$ is a product of distinct linear factors over the 
separable closure $k^s$ of $k$. We call such a pencil a {\it nonsingular} pencil.
 Let $X_Q$ be the base 
locus of the pencil $Q$, namely the intersection of the quadrics defined by 
$Q_1$ and $Q_2$ in $\P(V)$.  The variety  $X_Q$ is smooth. 
Let $T = Q_2Q_1^{-1}$. 
 
For a nonsingular quadratic form $q$ of rank $2n+2$ over $k$, we recall the definition of the  following groups:
$$\Sim(q) (k) = \{ \alpha \in \GL(V) \mid \alpha q \alpha^t = 
\lambda q  ~  \text{ for some }\lambda \in k^*  \} $$
$$\Sim^+(q) (k) = \{ \alpha \in \GL(V) \mid \alpha q \alpha^t = 
\lambda q  ~  \text{ for some }\lambda \in k^* ~\text{with} ~ \lambda^{n+1} = \det{\alpha}\} $$
$$\PSim^+(q) =   \Sim^+(q)/\Gm  .$$

We define 
$$\Sim_Q(k) = \{   \alpha \in \GL(V) \mid \alpha Q_1 \alpha^t = 
\lambda Q_1, \alpha Q_2\alpha^t = \lambda Q_2  ~\text{ with }~\lambda \in k^*\}.$$
Let  $\mu : \Sim_Q \to \Gm $  be the   similarity map given by $\mu(\alpha) = \lambda$.
We also define 
$$\Sim^+_Q(k) = \{   \alpha \in \GL(V) \mid \alpha Q_1 \alpha^t = 
\lambda Q_1, \alpha Q_2\alpha^t = \lambda Q_2  ~\text{ with }~\lambda \in k^*,  \lambda^{n+1} = \det(\alpha)\}$$
and
$\AAut^+_Q  =\Sim^+_Q/\Gm$. 

Since $H^1(k, \Gm)$ is trivial, we have a surjection $\Sim^+_Q(k) \to \AAut^+_Q(k)$.

We have   $\AAut^+_Q  \subset \PSim^+(Q_1) \cap \PSim^+(Q_2) \subset \PGL(V)$.

\begin{prop}\cite[Page 37]{Reid} 
\label{aut+}
\label{reed}
$\AAut^+_Q(k^s) \simeq  ({\mathbb Z}/2{\mathbb Z})^{2n}$. 
 \end{prop}

\begin{proof}
Let $\lambda_1, \cdots , \lambda_{2n+2}$ be the distinct zeroes of $f(t)$  in $k^s$. 
Let $v_i$ be a nonzero vector in the radical of the quadratic form $\lambda_i Q_1 - Q_2$, $ 1 \leq i \leq 2n+2$. Then the set of 
vectors $\{ v_1, \cdots , v_{2n+2} \}$ form an orthogonal basis for $Q_1$ and $Q_2$
 simultaneously. Further after scaling 
$v_i$, $ 1 \leq i \leq 2n+2$,  we may assume that $Q_1$ is represented by  
 the  identity matrix and $Q_2$ is the diagonal matrix with diagonal entries 
 $\lambda_1, \cdots , \lambda_{2n+2}$ with respect to the basis $\{v_1, \cdots  , v_{2n+2} \}$ of $V$. 
 
 Let $T \in \SL_{2n+2}(k^s)$ with $T Q_1 T^t = Q_1$ and $T Q_2 T^t = Q_2$.
 Then $T^{-1} = T^t$ and $T Q_2 = Q_2 T$. Since $Q_2$ is a diagonal matrix with distinct entries on the diagonal, 
 it follows that $T$ is a diagonal matrix with diagonal entries $\pm 1$. Since $T \in \SL_{2n+2}(k^s)$, 
 the product of all the entries  
 the diagonal of $T$ is 1.  Since every element in $\AAut^+_Q(k^s)$  admits a lift $T   \in \SL_{2n+2}(k^s)$  
 with $T Q_1 T^t = Q_2$ and $T Q_2 T^t = Q_2$,
 it follows that  
 $$\AAut^+_Q(k^s) = \{ \alpha \in \SL_{2n+2}(k^s)  \mid \alpha \text{ a diagonal 
 matrix with diagonal entries } \pm 1 \} / \{ \pm 1 \} $$ $$  \simeq  ({\mathbb Z}/2{\mathbb Z})^{2n}. $$ 
\end{proof}

The above proposition can also be deduced from  a  rational description of  the group scheme $\AAut^+_Q$ which is 
due to J.-P. Tignol. 

\begin{prop} Let $u  = Q_2Q_1^{-1} \in GL(V)$.  Then $k[u]$ is an \'etale $k$-algebra of dimension  $2n+2$
and there is an isomorphism of group schemes 
$  \AAut^+_Q \simeq R^1_{k[u]/k}(\mu_2)/\mu_2.$
\end{prop}

\begin{proof} Since $f(t)$ is the characteristic polynomial  of $u$  up to a non-zero scalar
 and $f(t)$    is a product of distinct linear factors over $k^s$, $k[u]$ is an \'etale algebra over $k$ of degree $2n+2$. 
 
 Let $\tau_1$  denote the adjoint  involution on $\End(V)$  with respect to $Q_1$.
 Then $\tau_1$ is identity on $k[u]$. 
Let $\alpha \in \Sim_Q(k)$. Then $\alpha$ commutes with $u$ and $\alpha \tau_1(\alpha) = \alpha^2 = \lambda$. Since $k[u]$ is a maximal commutative subalgebra 
 of $\End(V)$, $\alpha \in k[u]$. Thus 
 $$ \Sim_Q(k) = \{ \alpha \in k[u]^* \mid \alpha^2 \in k\}$$
and 
 $$ \Sim^+_Q(k) = \{ \alpha \in k[u]^* \mid \alpha^2  = \lambda \in k, N_{k[u]/k}(\alpha) = \lambda^{n+1}\}.$$
 We have exact sequences of group schemes 
 $$ 1 \to \Sim_Q/\Gm \to R_{k[u]/k}(\Gm)/\Gm \buildrel{2}\over{ \to} R_{k[u]/k}(\Gm)/\Gm \to 1 .$$
 This exact sequence yields an isomorphism 
 $$R_{k[u]/k}(\mu_2)/\mu_2 \simeq \Sim_Q/\Gm.$$
 
 The homomorphism $\theta:  \Sim_Q/\Gm \to \mu_2$ given by $\theta(\bar{\alpha}) =
 \det(\alpha)\mu(\alpha)^{-n-1}$ makes the following diagram commute 
 $$
 \begin{array}{ccc}
 R_{k[u]/k}(\mu_2)/\mu_2  & \to &  \Sim_Q/\Gm \\
 \Big\downarrow N_{k[u]/k} & &  \Big\downarrow \theta \\
 \mu_1 \hskip 12mm & = & \mu_2 ~~
 
 \end{array}
 $$
 Hence we have an isomorphism 
 $R^1_{k[u]/k}(\mu_2)/\mu_2 \simeq  \Sim^+_Q/\Gm = \AAut^+_Q$.
\end{proof}

Let $C$ be the smooth projective  hyperelliptic curve given  by the affine equation $y^2 = f(t)$
 and $\theta : C \to \P^1$ be the corresponding fibration.
 
 \begin{prop} 
 \label{autpic}There is a natural isomorphism  $ \theta_Q:  \Pictwo C \simeq \AAut^+_Q$.
 \end{prop}
 
 \begin{proof} We shall define this  map $\theta_Q$ over $k^s$ and prove that it descends to an isomorphism over $k$. 
 Let $\lambda_1, \cdots , \lambda_{2n+2}$ be the distinct zeroes of $f(t)$  in $k^s$.  Over $k^s$, we choose a basis 
  $\{ v_1, \cdots , v_{2n+2} \}$  of $V$  for which    $Q_1$ is represented by the identity matrix  and  $Q_2$  represented by 
   the diagonal matrix with diagonal entries 
 $\lambda_1, \cdots , \lambda_{2n+2}$ (cf. proof of \ref{reed}). 
 The points $\{\lambda_1, \cdots , \lambda_{2n+2} \} \in \A^1_{k^s}$ corresponds to the ramification
  points of $\theta : C \to \P^1$ over $k^s$. 
 Let   $P_1, \cdots , P_{2n+2} \in C(k^s)$   mapping onto the points 
 $\lambda_1, \cdots , \lambda_{2n+2} $ in $\A^1$ under $\theta$.
   Then $\Pictwo C$ is a 
 free ${\mathbb Z}/2{\mathbb Z}$-module generated  by  the divisor classes $\{  [P_{2n+2} - P_i]
  \mid 1 \leq i \leq 2n+1\}$ with exactly one 
 relation $\sum_1^{2n+1} [P_{2n+2} - P_i] = 0$. 
 Thus $\Pictwo C_{k^s} \simeq ({\mathbb Z}/2{\mathbb Z})^{2n}$. 
 Let  $ \theta_Q : \Pictwo C_{k^s} \to  \AAut_{Q_{k^s}}^+$ be the map
 $\theta_Q([P_{2n+2} - P_i]) = <\epsilon_1, \cdots , \epsilon_{2n+2}> $ modulo $\mu_{2n+2}$, with $\epsilon_{2n+2} = \epsilon_i = -1$ and 
 $\epsilon_j = 1$ for $j \neq i, 2n+2$. Then $\theta_Q$ is an isomorphism of groups. 
   This map is independent of the choice of basis $\{v_1, \cdots v_{2n+2} \}$ of $V$ over $k_s$.
   In fact any such choice of a basis $\{ w_1, \cdots , w_{2n+2} \}$ of $V$ with $Q_1$ represented by the identity matrix and 
   $Q_2$ is  represented by the diagonal matrix $<\lambda_1, \cdots , \lambda_{2n+2}>$ 
   necessarily is given by $w_j = \epsilon_j v_j$, $\epsilon_j = \pm 1$, 
   $1 \leq j \leq 2n+2$. Hence  $\theta_Q([P_{2n+2}  - P_i])$ is well defined, commutes with 
   the Galois action and descends to a $k$-isomorphism
  $ \theta_Q:  \Pictwo C \simeq \AAut^+_Q$.  
  \end{proof}
 
 \section{Twisting of Pencils} 
 \label{twisting-pencils}
Let $k$ be a field  of characteristic not 2.
Let $Q$ be a nonsingular  quadratic form of dimension $2n + 2$ over $k$ represented 
by a symmetric matrix $q \in \GL_{2n+2}(k)$. Let 
$\Gamma_k$ be the Galois group $k^s/k$ and 
$\gamma : \Gamma_k\to \PSOrth(q)(k^s)$ be a 1-cocycle. Then the twisted quadric $q^{\gamma}$ 
as  a variety over $k$ is defined 
by Tao (\cite{T}).  Let $A$ be a central simple algebra over $k$ of degree $2n+ 2$ with an orthogonal
 involution $\tau$ associated to the cocycle
$[\gamma] \in H^1(k, \PSOrth(q))$ (\cite[p. 409]{KMRT}). The {\it twisted quadric} $q^\gamma$ 
is the variety of $(2n+2)$-dimensional  totally isotropic 
ideals in $(A, \tau)$. If $A = M_{2n+2}(k)$ and $\tau$ is adjoint to a quadratic form $\tilde{q}$ 
of dimension $2n+2$, the twisted quadric $q^\gamma$ is indeed the 
quadric defined by $\tilde{q}$. 

Let $Q_1$  and $Q_2$ be  two  nonsingular quadratic forms defined on a vector space $V$ of  dimension  $2n+2$.  
Let $Q  = tQ_1 - Q_2$ be the associated pencil in $\P(V)$. We assume that $Q$ is nonsingular pencil.
For a choice of a basis of $V$, we assume that 
$Q_1$ and $Q_2$ are symmetric matrices.
Given a cocycle 
$\alpha  : \Gamma_k\to \AAut^+(Q)$, we shall describe the twisted pencil $Q^\alpha$. 
 For $\sigma \in \Gamma_k$,  let $\alpha_\sigma \in \SL_{2n+2}(k^s)$ such that 
$\alpha(\sigma) = \alpha_\sigma$ modulo $\mu_{2n+2}$.  We have 
$\alpha_\sigma Q_1 \alpha_\sigma^t = \lambda_\sigma Q_1$ and 
$\alpha_\sigma Q_2 \alpha_\sigma^t = \lambda_\sigma Q_2$ with $\lambda_\sigma \in k^{s*}$ and 
$\lambda_\sigma^{n+1} = 1$.  For $\sigma, \tau \in \Gamma_k$,
 let $\epsilon_{\sigma,\tau} = \alpha_\sigma \sigma(\alpha_\tau) (\alpha_{\sigma\tau})^{-1}  \in \mu_{2n+2}$. 
Then $\epsilon_\alpha : \Gamma_k \to \mu_{2n+2} \subset \GL_{2n+2}$ given by $\epsilon_\alpha(\sigma, \tau) = \epsilon_{\sigma, \tau}$ 
 is a 2-cocycle which defines a central simple algebra $A$ over $k$ of degree $2n+2$.

\noindent{\bf Case 1:}  Suppose $[\epsilon_\alpha] \in H^2(k, \mu_{2n+2})$ is zero.  Then  $A = M_{2n+2}(k)$.  There exists 
a function $g : \Gamma_k \to \mu_{2n+2}$ such that $\epsilon_\alpha(\sigma, \tau) = g(\sigma)~\! ^\sigma g(\tau) g(\sigma\tau)^{-1}$. 
Then $g(\sigma)^{-1}\alpha_\sigma$ is a 1-cocycle with values in $\Sim(Q_1) \cap \Sim(Q_2) $.
We have 
$$ (g(\sigma)^{-1} \alpha_\sigma) Q_1 (g(\sigma)^{-1}\alpha_\sigma)^t = g(\sigma)^{-2} \lambda_\sigma Q_1$$
and
$$ (g(\sigma)^{-1} \alpha_\sigma) Q_2 (g(\sigma)^{-1}\alpha_\sigma)^t = g(\sigma)^{-2} \lambda_\sigma Q_2$$
with $g(\sigma)^{-2}\lambda_\sigma \in \mu_{2n+2}$ is a 1-cocycle. 
Since $H^1(k, G_m)$ is zero,
 there is  $\theta \in k^{s*}$ such that $g(\sigma)^{-2} \lambda_\sigma 
= \theta^{-1} \sigma(\theta)$ for $\sigma \in \Gamma_k$. 
Let $\theta = \delta^2$  with $\delta \in k^{s*}$. Then 
$(\delta^{-1} \sigma(\delta))^{2n+2} = (\theta^{-1} \sigma(\theta))^{n+1} = 1$.
The map $\Gamma_k \to \SL_{2n+2}(k^s)$ given by $\sigma \mapsto (\delta\sigma(\delta)^{-1})
 g(\sigma)^{-1} \alpha_\sigma$ is a 1-cocycle with values in 
$\SOrth(Q_1) \cap \SOrth(Q_2)$. This cocycle yields descents $\tilde{Q}_1$ and $\tilde{Q}_2$ 
of the quadratic forms $Q_1$ and $Q_2$ with 
det$(\tilde{Q}_1) =$ det$(Q_1)$ and det$(\tilde{Q}_2) =$ det$(Q_2)$. Let $\tilde{Q} = t\tilde{Q}_1 - \tilde{Q}_2$.

There is an isomorphism 
$$\psi : Q_{k^s} \to \tilde{Q}_{k^s}$$
such that the associated cocycle $\Gamma_k \to \AAut^+_Q(k^s)$ defined by $\sigma \mapsto \psi^{-1}
 \sigma\psi \sigma^{-1}$ is precisely 
$(\delta \sigma(\delta))^{-1} g(\sigma)^{-1} \alpha_\sigma$. Here $\sigma$ stands for the 
Galois action on $Q_{k^s}$ as well $\tilde{Q}_{k^s}$. 
We call $t\tilde{Q}_1 - \tilde{Q}_2$ the {\it twist} of $tQ_1 - Q_2$ with respect to the cocycle $\alpha$. 

\vskip 2mm

\noindent{\bf Case 2:}  Suppose $[\epsilon_\alpha] \in H^2(k, \mu_{2n+2})$ is  non-zero.
The map $\Gamma_k \to \PSOrth(Q_1)$, $\sigma \mapsto $int$(\alpha_\sigma)$ 
is a 1-cocycle. This yields a descent $(A, \tau_1)$  of $(M_{2n+2}(k^s), \tau_{Q_1})$ where 
$A$ is a central simple algebra of degree $2n+2$ and $\tau_1$ an orthogonal involution on $A$. Here 
$\tau_{Q_1}$ is the adjoint involution $X \mapsto Q_1X^tQ_1^{-1}$ on the matrix algebra.
We note that $A$ is the  descent of the cocycle $\alpha_\sigma \in \PGL_{2n+2}(k^s)$. 
There is an isomorphism of algebras with involutions 
$$\phi : (M_{2n+2}(k^s), \tau_{Q_1})_{k^s}  \to  (A, \tau_1)_{k^s}$$
such that int$(\alpha_\sigma) = \phi^{-1} \sigma \phi \sigma^{-1}$.

   Since 
$$\tau_{Q_1}(Q_2Q_1^{-1}) = Q_1 (Q_2Q_1^{-1})^t  Q_1^{-1} = Q_2Q_1^{-1},$$
$Q_2Q_1^{-1}$ is a symmetric element in $(M_{2n+2}(k), \tau_{Q_1})$.   
For $\sigma \in \Gamma_k$,  we prove that $\sigma\phi(Q_2Q_1^{-1}) = \phi(Q_2Q_1^{-1})$ thereby proving that 
$\phi(Q_2Q_1^{-1}) \in A$.  We have 
$$
\begin{array}{rcl}
\sigma(\phi(Q_2Q_1^{-1})) & = & \phi ~int(\alpha_\sigma) \sigma(Q_2Q_1^{-1}) \\
& = &  \phi ((\alpha_\sigma Q_2 \alpha_\sigma^t) ((\alpha_\sigma^t)^{-1}Q_1^{-1} \alpha_\sigma^{-1})) \\
& = & \phi(\lambda_\sigma Q_2) \lambda_\sigma^{-1} Q_1^{-1}) \\
& = & \phi(Q_2Q_1^{-1}). 
\end{array}
$$
Thus $u = \phi(Q_2Q_1^{-1})$ is a symmetric element in $(A, \tau_1)$.  The pencil  $t - u$ on $(A, \tau_1) \otimes_kk[t]$ 
is called the {\it twisted} pencil associated to $\alpha$.

Suppose $A = M_{2n+2}(k)$ and  $\tau_1 = \tau_{\tilde{Q}_1}$. Let $\tilde{Q} = t\tilde{Q}_1 - \tilde{Q}_2$ be the descent of 
$tQ_1 - Q_2 = Q$ and 
$\psi : Q_{k^s} \to \tilde{Q}_{k^s}$ an isomorphism 
satisfying  $\psi^{-1} \sigma \psi \sigma^{-1} = \alpha_\sigma \in \PSOrth(Q_1) \cap \PSOrth(Q_2)$.
Then 
$$
int(\psi) = \phi : (M_{2n+2}(k^s), \tau_{Q_1})  \to (A, \tau_{\tilde{Q}_1})_{k^s}
$$
is an isomorphism  with 
$$\phi^{-1} \sigma \phi \sigma^{-1}  = int(\alpha_\sigma).$$
It is easy to  verify that 
$$\phi(Q_2Q_1^{-1}) = \tilde{Q}_2\tilde{Q}_1^{-1} = u$$
and the pencil $t - u$  over $(M_{2n+2}(k), \tau_{\tilde{Q}_1})$ corresponds under Morita 
equivalence to the pencil $(t-u)(\tilde{Q}_1) = t\tilde{Q}_1 - \tilde{Q}_2$ over $k$. 

\begin{defin} The {\it Grassmannian}  $I_s(Q^\alpha)$ of $s$-dimensional zero subspaces of the twisted pencil $Q^\alpha$ is the 
Grassmannian of $s(2n+2)$-dimensional totally isotropic ideals in $A$ for both the involutions $\tau_1$
and $Int(u) \circ \tau_1$.  
\end{defin}

If $A$ is split and $Q^\alpha = t\tilde{Q}_1 - \tilde{Q}_2$, $I_1(Q^\alpha)$ is simply the base locus $X$ of the
intersection of $\tilde{Q}_1$ and $\tilde{Q}_2$.

\begin{prop} 
\label{trivial-twist}
Let $[\alpha_\sigma] \in H^1(k,  \AAut^+(Q))$ and $Q^\alpha$ the twisted pencil associated to $[\alpha_\sigma]$.
Then $Q^\alpha$ is a pencil of quadrics if and only if  the image of $[\alpha_\sigma]$ in $H^1(k, \PGL_{2n+2})$ under the 
composite $\AAut^+(Q) \to \PSOrth(Q_1) \to \PGL_{2n+2}$ is zero.  Further if $Q^\alpha$ is a pencil of quadrics, then
disc$(Q^\alpha) = $ disc$(Q)$. 
\end{prop} 

\begin{proof}
The only remark is that for a symmetric unit $u$ in the algebra  with involution 
$(A, \tau)$ over $k$,
the pencil $(t - u)$ on the central simple algebra $A \otimes k[t]$ with the involution $\tau \otimes 1$ is a 
pencil of quadrics if and only if $A$ is split.  Since the cocycle $\alpha_\sigma$ has values in $\AAut^+(Q)$, it follows that 
disc$(Q^\alpha) = $ disc$(Q)$.
\end{proof}

\section{Twisted moduli spaces on hyperelliptic curves}
\label{twisting-moduli-hyperell}
Let $C$ be a smooth projective geometrically integral 
hyperelliptic curve of genus $g \geq 2$ 
defined over a field $k$ of characteristic not 2 and   $\theta: C \to {\mathbb P}^1$ be the hyperelliptic
covering.  We assume that  there is a rational point  $P \in  C(k)$ unramified for $\theta$.
 Let $Q = \theta(P)  \in \P^1(k)$.
Let  $\eta = {\OO}((2g-1)P)$ and   $M_2(C,\eta)$   the moduli space of  rank two stable locally free sheaves with  
 determinant $\eta$ on $C$. 
 We recall a construction of  Ramanan (\S \ref{Appendix}) identifying the moduli space $M_C(2, \eta)$ of rank two 
 stable locally free sheaves of determinant  $\eta$ with a certain quadratic Grassmannian. The above construction 
 over an algebraically closed field goes back to (\cite{De-Ra}). 
 
 Let $D$ be the divisor on ${\mathbb P}^1$ of degree $2g + 2$ with support the ramification points of $\theta$ and 
 $\tilde{D}$ the divisor on $C$ of degree $2g+2$ mapping isomorphically onto $D$. Let 
 $s$ be a general section in $H^0(\P^1, D)$. Restricted to $\A^1 = \P^1 -\{ Q\}$, $s$ is given by 
 a polynomial $f(t)$ of degree $2g + 2$. Further, $Q$ being a rational point of $\P^1$ which splits into two rational 
 points $P$ and $P'$ in $C(k)$, leading coefficient of $f(t)$ is a square which we assume to be one after scaling $s$.
 The affine equation of $\theta$ restricted to $\A^1 = \P^1 - \{ Q\}$ is
 given by $y^2 = f(t)$ with $f(t)$ monic and has distinct zeroes $\lambda_1, \cdots , \lambda_{2n+2} $ in $k^s$ corresponding to 
 the support of $D$ over $k^s$.

 Let $\eta  = {\OO}((2g - 1)P)$.   Then  $V =  H^0(\tilde{D}, \eta\mid_{\tilde{D}})$ is a vector space 
 of dimension $2g + 2$. Associated to $\eta$, there is a nonsingular pencil  
 $Q_\eta$ in $V$ (\ref{pencil-eta}) 
 with disc$(Q_\eta) = \lambda f(t)$ for some scalar $\lambda$.  The pencil $Q_\eta$ at  the point $Q$ is split  so that 
 disc$(Q_\eta)_{t = \infty} = 1$.  Thus $\lambda$ is a square and after scaling we assume disc$(Q_\eta) = f(t)$ with 
 $f(t)$ monic.

Let   $M_C(2,\eta)$  be   the moduli space of  rank two stable locally free sheaves with  
 determinant $\eta$ on $C$. Let $I_{g-1}(Q_\eta)$ be the Grassmannian of $(g-1)$ dimensional 
 linear subspaces contained in the base locus $X_\eta$ of $Q_\eta$. 
 Then an isomorphism 
 
\hspace*{\fill}
 $ \phi_\eta : M_C(2, \eta) \to I_{g-1}(Q_\eta)$   \hfill   $(\star \star \star) $
\hspace*{\fill}

\noindent
 over $k$ is constructed in (\ref{app-main}).
 In fact this isomorphism commutes with the action of $\Pictwo(C)$ on both sides (\ref{actions}).   
 We shall now discuss a twisted version of this isomorphism.

Let $\alpha \in H^2(C, \mu_2)$ be a $\mu_2$-gerbe. Let $M_C^\alpha(2, \eta)$ be the  moduli 
space of $\alpha$-twisted rank 2 stable locally free sheaves with determinant $\eta$  over $C$.
 Let  $\beta : \Gamma_k  \to \twoPic C(k^s)$ be a 1-cocycle with cocycle class 
 $[\beta]$ being the image of $\alpha$ under projection in $(\star \star)$ of \S 3. 
 By  \S\ref{moduli-stacks}, $M_C^\alpha(2, \eta)$ is the Galois twist of $M_C(2, \eta)^\beta $ of the moduli space
 $M_C(2, \eta)$ by the cocycle $\beta$.
 
 On the other hand,  the isomorphism 
 $ \theta_{Q_\eta} : \twoPic (C) \to \AAut^+(Q_\eta)$ defined in (\ref{autpic})  gives a cocycle 
 $\theta_{Q_\eta} \circ \beta$ with values in $\AAut^+(Q_\eta)$. We denote by $Q_\eta^\alpha$ the twisted 
 pencil associated to the 1-cocycle $\theta_{Q_\eta}\circ \beta$. Then 
 $I_{g-1}(Q_\eta^{\alpha})$ is simply a descent  of $I_{g-1}(Q_\eta)$ with respect to $\theta_{Q_\eta}\circ \beta$; i.e.
 $$
 I_{g-1}(Q_\eta^\alpha) = I_{g-1}(Q_\eta)^\beta.$$
Twisting both sides of $(\star \star \star)$ by $\alpha$, we get a $k$-isomorphism  
$$\phi_\eta^\alpha :  M_C^\alpha(2, \eta) \to I_{g-1}((Q_\eta)^\alpha).$$

We record this in the following 

\begin{theorem}
\label{phi-eta} Let $C$ be a smooth projective geometrically integral 
hyperelliptic curve of genus $g \geq 2$  defined over a field $k$ of characteristic not 2 and   $\theta: C \to \P^1$ be a
double cover.  Suppose that  there exists $P \in  C(k)$ unramified for $\theta$.
 Let  $\eta = {\OO}((2g-1)P)$ and   $\alpha \in H^2(C, \mu_2)$  a $\mu_2$-gerbe. 
 Then there is a $k$-isomorphism 
$$\phi_\eta^\alpha :  M_C^\alpha(2, \eta) \to I_{g-1}(Q_\eta^\alpha)$$
which descends the isomorphism 
$\phi_\eta  :  M_C (2, \eta) \to I_{g-1}(Q_\eta)$. 
\end{theorem}

\section{Soluble pencils}
\label{soluble}
 
 Let $k$ be a field  of characteristic not 2.
 A  nonsingular  pencil $Q = tQ_1 - Q_2$  on a vector space $V$ of dimension $2n+2$
  is {\it soluble} if there is   subspace of $V$ of dimension $n$  on
   which both $Q_1$ and $Q_2$ are identically zero; in other words, the base locus $X_Q$ contains 
   an $n$-dimensional linear subspace.  
   
   Let $C$ be a smooth projective geometrically integral 
hyperelliptic curve of genus $g \geq 2$ 
defined over   $k$  and   $\theta: C \to \P^1$ be a
double cover.  Suppose that  there exists $P \in  C(k)$ unramified for $\theta$.
 Let $Q = \theta(P)  \in \P^1(k)$.
 Choosing $\A^1 = \P^1 - \{ Q \}$, we may write the affine equation of $C$ by 
 $$y^ 2= f(t)$$
 with $f(t)$ monic of degree $2g+2$. 
 Let  $\eta$ be the line bundle on $C$ given by $\eta =  {\OO}((2g-1)P)$. 
 Let $\tilde{D}$ the divisor on $C$ mapping isomorphically onto the divisor $D$ on $\P^1$ which is
 the ramification divisor for $\theta$.  Then $D \subset \A^1$ and  the support of $ D $ is precisely 
 the  set zeros of $f(t)$. Let $V  = H^0(\tilde{D}, \eta \mid_{\tilde{D}})$.  The dimension of  
$V$ is  $2g+2$ and there is a pencil of quadrics $Q_\eta$ in $\P(V)$ (\S\ref{Appendix}) satisfying 
 the following properties.
 
 \begin{prop}
 \label{qeta}
1) The pencil $Q_\eta$ is soluble.\\
2) Suppose $Q_\eta = tQ_\eta^1 - Q_\eta^2$. Then 
$Q^1_\eta$ is hyperbolic and disc$(Q_\eta) = f(t)$.  
  \end{prop}
 
 \begin{proof}
 1) By  (\S  \ref{pencil-eta})  there is a $g$-dimensional linear subspace in $V$ on which 
 $Q_\eta^1$ and $Q_\eta^2$ vanish identically. Thus $Q_\eta$ is a soluble pencil. 
 
 \vskip 3mm
 
 \noindent
 2) By  (\S  \ref{pencil-eta})  disc$(Q_\eta) = \lambda f(t)$ for some $\lambda \in k^*$ and 
 $(Q_\eta)_Q$ is hyperbolic. Since $(Q_\eta)_Q = Q^1_\eta$, disc$(Q_\eta^1) = $ square class of 
 leading coefficient of $f(t)$ $= 1$. Thus $\lambda$ is a square  and disc$(Q_\eta) = f(t)$. 
 \end{proof}  
 
 There is another example of a soluble pencil constructed in (\cite{Wang}).
 We recall this construction. 
 Let $f(t) \in k[t]$ be a monic polynomial of degree $2n + 2$ with no multiple zeros.  
 Then  $L = k[t]/(f(t))$ is an \'etale  $k$-algebra  of dimension $2n+2$. 
 Let  $\beta \in L $  be  the class of $t$. Then
$\{ 1, \beta, \cdots , \beta^{2n+1} \}$ is a basis of $L$ over $k$.
We define a bilinear form $Q_1^0 : L \times L \to k$ by 
 $$Q_1^0(\lambda, \mu) = {\rm ~coefficient ~ of~ } \beta^{2n+1} {\rm~ in~ } \lambda\mu.$$
 Then $Q_1^0$ is hyperbolic. 
 Let $T_0 : L \to L$ denote the linear map given by multiplication by $\beta$. Then $T_0$ is a self 
 adjoint operator with respect to $Q_1^0$ and the characteristic polynomial of $T_0$   is $f(t)$.
 Let $Q_2^0 = Q_1^0T_0$.  Then $Q_0 =  tQ_1^0  -  Q_2^0$ is a nonsingular pencil with
  discriminant  $f(t)$.   Since $Q_1^0$ and $Q_2^0$ are identically zero 
on  the $n$-dimensional  subspace generated by $\{1, \beta, \cdots \beta^{n-1}\}$ of $L$, 
   $Q_0$ is a soluble pencil.

 Let $C \to \P^1$ be the smooth projective  hyperelliptic curve given  by the affine equation $y^2 = f(t)$
and   $ \theta_{Q_0}:  \Pictwo C \simeq \AAut^+_{Q_0}$ the isomorphism given in (\ref{autpic}).
Let $c : \AAut^+(Q_0) \to \PSOrth(Q_1^0)$ be the inclusion. 
 Then   
 
 \hspace*{\fill}
$ \frac{\Pic C}{2} \buildrel{\delta}\over{\to}  H^1(k, \Pictwo C) \buildrel{c\theta_{Q_0}}\over{\to}  H^1(k, \PSOrth(Q_1^0))$
 \hfill  $(\#)$
 \hspace*{\fill}
 
 \noindent
is a complex (\cite[Theorem 10]{SW}).

\begin{prop}
\label{complex}
 Let $f(t) \in k[t]$ be a monic polynomial of degree $2n + 2$ with no multiple zeros. 
Let $C \to \P^1$ be the smooth projective  hyperelliptic curve given  by the affine equation $y^2 = f(t)$. 
Let $Q = tQ_1 - Q_2$ be a soluble pencil with disc$(Q) = f(t)$. 
Let $ \theta_Q:  \Pictwo C \simeq \AAut^+_Q$ be the isomorphism given in (\ref{autpic}).
Then the following sequence 

\hspace*{\fill}
 $ \frac{\Pic C}{2} \buildrel{\delta}\over{\to}  H^1(k, \Pictwo C) \buildrel{c\theta_Q}\over{ \to} H^1(k, \PSOrth(Q_1))$ \hfill  $(\star \star \star \star)$
 \hspace*{\fill}
 
 \noindent
 is a complex. 
\end{prop}

We begin with the following 

\begin{lemma}
\label{complex-lemma}Let $Q = tQ_1 - Q_2$ be a soluble pencil with disc$(Q) = f(t)$ a monic polynomial 
with no multiple zeros. With the notation as in (\ref{complex}), there is a 1-cocycle 
$\beta : \Gamma_k \to \Pictwo C$ whose class in $H^1(k, \Pictwo C)$ belongs to the image of 
$\delta : \frac{\Pic C}{2} \to H^1(k, \Pictwo C) $ such that $Q = Q_0^\beta$. 
\end{lemma} 

\begin{proof} Since $Q$ is soluble, $Q_1$ contains an $n$-dimensional linear subspace. We have
disc$(Q) = $ disc$(tQ_1 - Q_2) = f(t)$ with $f(t)$ monic. Hence disc$(Q_1) = $ leading coefficient of 
$f(t)$  modulo squares  = 1. Hence $Q_1$ is hyperbolic and $Q_1$ is isomorphic to $Q_1^0$. Hence 
$Q$ is isomorphic to $tQ_1^0 - Q_1'$ for some  regular quadratic form $Q_1'$. The pencil $tQ_1^0 - Q_1'$ is soluble and in view of 
(\cite[Theorem 10]{SW}), there is $[N] \in \frac{\Pic C}{2}$ whose image $\tilde{\beta} \in H^1(k, \Pictwo C)$ under $\delta$ gives the
twist $(Q_0)^\beta \simeq tQ_1^0 - Q_1'$; here $\beta$ is a cocycle representing $\tilde{\beta}$.
Since $Q  \simeq tQ_1^0 - Q_1'$, we get $Q \simeq (Q_0)^\beta$. 
\end{proof}

We now recall the twisting bijection in Galois cohomology (\cite{Serre}). Let $G$ be a discrete 
$\Gamma_k$-group and $a : \Gamma_k \to G$ be a cocycle. 
Let $_aG $ be the twist of $G$ by the cocycle $\tilde{a}: \Gamma_k \to \Aut(G)$ given by 
$\tilde{a}(\sigma)(g) = a(\sigma)\sigma(g) a(\sigma)^{-1}$. 
The twisting mp $_aT : H^1(k, ~_aG) \to H^1(k, G)$ is defined by $_aT(b) = ba$ at the level of cocycles. 
Indeed one verifies that if $b : \Gamma_k \to ~_aG$ is a 1-cocycle, the map 
$\Gamma_k \to G$, $\sigma \mapsto b(\sigma)a(\sigma)$ is a 1-cocycle. The twisting map
$_aT$ is a bijection which sends $[1] $ to $[a] \in H^1(k, G)$. 

 \vskip 3mm
 
 \noindent
{\it Proof of the proposition:}  Since $Q = tQ_1 - Q_2$ is a soluble pencil with disc$(Q) = f(t)$, by (\ref{complex-lemma}), 
 there is $\alpha \in \frac{\Pic C}{2}$ such that if $\tilde{\beta} = \theta_{Q_0} \circ \delta(\alpha) \in H^1(k, \AAut^+(Q_0))$, 
 a cocycle $\beta$ associated to $\tilde{\beta}$ gives an isomorphism $(Q_0)^\beta \simeq Q$. Hence 
 $\AAut(Q_0)^\beta = \AAut(Q)$, $(\PSOrth(Q_1^0))^{\beta} = \PSOrth(Q_1)$. We have a commutative diagram of maps 
  $$
   \begin{array}{cccccc}
     H^1(k, \Pictwo C) &  \buildrel{\theta_{Q}}\over{ \to} & H^1(k, \AAut^+(Q)) &   \buildrel{c}\over{ \to} & H^1(k, \PSOrth(Q_1))  \\
 _{\delta( \alpha)}T \Big \downarrow   &  & _{\beta}T \Big\downarrow   & & _{c\beta}T \Big \downarrow  \\
    H^1(k, \Pictwo C) &  \buildrel{\theta_{Q_0}}\over{ \to} & H^1(k, \AAut^+(Q_0)) &   \buildrel{c}\over{ \to} & H^1(k, \PSOrth(Q^0_1)) . 
   \end{array}
   $$
 Here $(\Pictwo C)^\delta(\alpha) = \Pictwo C$ and $_{\delta(\alpha)}T(f) = f + \delta(\alpha)$ for all $f \in H^1(k, \Pictwo C)$. 
 
 We now prove that the sequence 

\hspace*{\fill}
 $ \frac{\Pic C}{2}\buildrel{\delta}\over{\to}  H^1(k, \Pictwo C) \buildrel{c\theta_Q}\over{ \to} H^1(k, \PSOrth(Q_1))$
 \hspace*{\fill}
 
\noindent
 is a complex. 
 
 Let $\gamma \in \frac{\Pic C}{2}$.  Then 
 $$
 _{c\beta} T c \theta_Q \delta(\gamma) = c \theta_{Q_0} ~_{\delta(\alpha)}T (\delta(\gamma)) = 
 c \theta_{Q_0}  (\delta(\alpha + \gamma)) = 0$$
 in view of $ (\#)$. 
 
 Since $Q$ is a soluble pencil with disc$(Q) = f(t)$ monic, $Q_1$ is hyperbolic and
 $Q_1 \simeq Q_1^0$. Since $Q_1 = (Q_1^0)^{c\beta}$, $c\beta = 0 \in H^1(k, \PSOrth(Q_0))$ and
 $_{c\beta}T$ is a bijection. Thus,
 $(c \theta_Q) \circ \delta $ is zero and $(\star \star \star \star)$ is a complex. 
 \hfill $\Box$.

 \begin{cor}
 \label{mc-ig1} Let $k$ be a number field and $C/k$ a smooth projective geometrically integral 
 hyperelliptic curve of genus $g \geq 2$. 
 Suppose that $C$ admits a rational point $P$ which is unramified for the covering 
 $C \to \P^1$. Let $\eta = \OO((2g-1)P)$ and $Q_\eta$ the pencil of quadrics in (\S \ref{Appendix}).
 Let $\alpha \in H^2(C, \mu_2)$ be a gerbe which is locally  essentially trivial. Then $Q_\eta^\alpha$ 
 is a pencil of quadrics and the twisted moduli space $M_C^\alpha(2, \eta)$ is isomorphic to the 
 Grassmannian $I_{g-1}(Q_\eta^\alpha)$ of the pencil of quadrics $Q_\eta^\alpha$.  
 \end{cor}
 
 \begin{proof} Since $\alpha$ is locally essentially 
 trivial, $\alpha \in Im(\delta_{k_\nu} : \frac{\Pic C_{k_\nu}}{2} \to H^2(C_{k_\nu}, \mu_2))$ for all places $\nu$ of $k$. 
 Let $\beta \in H^1(k, \Pictwo C)$ be the image of $\alpha$.  
 Then $\beta \in  Im (\frac{\Pic C_{k_\nu}}{2} \to H^1(k_\nu, \Pictwo C))$ for all places $\nu$ of $k$. 
 By (\ref{complex}), we have $c\theta_{Q_\eta}(\beta) = 0 \in H^1(k_\nu, \PSOrth(Q_\eta^1))$ for all places $\nu$ of $k$. 
 In particular the image of $c\theta_{Q_\eta}(\beta)$ in $H^1(k_\nu, \PGL_{2g+2})$ is trivial  for all places $\nu$ of $k$. 
 Let $A$ be a central simple algebra representing the image of $c\theta_{Q_\eta}(\beta)$ in $H^1(k, \PGL_{2g+2})$.
 Then $A \otimes {k_\nu}$ is a matrix algebra for all places $\nu$ of $k$. 
 By Hasse-Brauer-Noether-Albert theorem (\cite{H-B-N}),  $A$ is a matrix algebra. 
 Hence by (\ref{trivial-twist}), $Q_\eta^\alpha$ is a pencil. 
 \end{proof}

\section{Pencils of quadrics in $\mathbb{P}^5$ over number fields }
\label{pencil-p5}

In this section we prove the following weak form of a Hasse principle for a smooth intersection of two quadrics in 
$\mathbb{P}^5$ over a number field.

\begin{theorem} 
\label{lgp}
Let $k$ be a global field of characteristic not $2$. Let $X$ be a smooth intersection of two quadrics 
$Q_1$ and $Q_2$ in $\mathbb{P}^5$.  Suppose $disc(Q_1) = 1$. If $X(k_\nu)$ contains a line for all places $\nu$ of $k$, 
then, $X(k)$ is nonempty.
\end{theorem} 

\begin{proof}
Let $Q =t Q_1 - Q_2$ and $f(t) = disc(Q)$.  Since $X = Q_1 \cap Q_2$ is smooth,  $f(t)$ has distinct linear factors.
 Since $disc(Q_1) = 1$,  $f(t)$ 
has leading coefficient a square, which we assume to be 1 after scaling.  Further the degree of $f(t)$ is 6. 
Let $C$ be the smooth projective  hyperelliptic curve given  by the affine equation $y^2 = f(t)$
 and $\gamma : C \to {\mathbb P}^1$ be the corresponding covering.  Since $f(t)$ is monic, the point $Q$ at infinity of 
 $\mathbb{P}^1$ splits into two rational points in $C(k)$. Let $P$ be a point of $C(k)$ mapping to $Q$ under $\gamma$.

 Let $Q_{\eta}$ be the pencil of quadrics in $\mathbb{P}^5$ constructed by Ramanan in  (\S\ref{Appendix}) with reference to the 
 choice of the line bundle $\eta = \OO((2g-1)P)$ (cf. \S  5). Then $ Q_{\eta}$ is a soluble pencil with 
 disc($Q_{\eta} $)$= f(t)$. Hence there is an isomorphism  $Q_{\eta,k^s} \simeq Q_{k^s}$.
 Let $\beta : \Gamma_k \to \Pictwo C$ be a cocycle   such that $Q = Q_{\eta}^{\beta}$. 
  Let the class  $\tilde{\beta} \in H^1(k, \Pictwo C)$ of $\beta$ be the image of 
  $\alpha \in H^2(C, \mu_2)$ with $\alpha_{k^s} = 0$. We claim that $\alpha$ is locally essentially trivial.

 By hypothesis, the pencil $Q_{k_{\nu}}$ is soluble with disc ($Q_1)$ = 1.
  By (\ref{complex-lemma}), there is  a 1-cocycle $\theta_{\nu}: \Gamma_{k_{\nu}} \to 
 \Pictwo C_{k_\nu}$  whose class  belongs to the image of $\delta: Pic C_{k_{\nu}} /2 \to H^1(k_{\nu}, \Pictwo C)$
 such that $Q_{k_\nu} = Q_0^{\theta_\nu}$, where $Q_0 $ is the pencil of 
 Arul-Wang (cf. \S 5).  Since $Q_{\eta}$ is also a soluble pencil with $Q_{\eta}^1$ hyperbolic, there is a 1-cocycle  $\phi: \Gamma_k \to 
 \Pictwo C$ whose class belongs to the image of $\delta$ and $Q_0 = Q_\eta^\phi$.
 Thus over $k_{\nu}$,  $Q = Q_{\eta}^{\theta_{\nu} + \phi}$. Since $Q = Q_{\eta}^{\beta}$, we have, locally,
 $[\beta] = [\theta_v + \phi]$ which is in the image of $\delta: Pic C_{k_{\nu}} /2 \to H^1(k_{\nu}, \Pictwo C)$.  Thus
 $\alpha$ maps to zero in $H^2(C_{k_\nu}, G_m) $ and  $\alpha$ is locally essentially trivial. 
 
 We have an isomorphism (\ref{mc-ig1})
 
 $$\phi_\eta^\alpha :  M_C^\alpha(2, \eta) \to I_{1}(Q_\eta^\alpha ) =     I_{1}(Q)  = X.$$
 
 Since $\alpha$ is locally essentially trivial  and degree($\eta$) is odd, by (\ref{essentially-trivial}), 
  $M_C^\alpha(2, \eta)(k)$ is nonempty. Since  $\phi_\eta^\alpha$ is an isomorphism, $X(k)$
  is not empty and this completes the proof of the theorem.
\end {proof}
 
  Let $Q = tQ_1 - Q_2$ be a pencil of quadrics 
in $\mathbb{P}^5$ with disc$(Q) = f(t)$. Let $C$ be the smooth projective  hyperelliptic  curve $C$ given by $y^2 = f(t)$. 
The next corollary, as pointed out by Colliot-Th\'el\`ene,  replaces the  condition disc$(Q_1) = 1$ in (\ref{lgp})  with 
 $C(k) \neq \emptyset$. 
  
 \begin{cor} 
\label{lgp-cor}
Let $k$ be a global field of characteristic not $2$. Let $Q = tQ_1 - Q_2$ be a pencil of quadrics 
in $\mathbb{P}^5$ with disc$(Q) = f(t)$ and  $X = Q_1 \cap Q_2$ smooth.  
Suppose the hyperelliptic  curve $C$ given by $y^2 = f(t)$ has a  zero-cycle of degree 1.
 If $X(k_\nu)$ contains a line for all places $\nu$ of $k$, 
then, $X(k)$ is nonempty.
\end{cor} 

\begin{proof} The curve $C$ having a rational point which is unramified for the hyperelliptic covering $C \to {\mathbb P}^1$
is equivalent to writing the pencil $Q$ as $tQ_1 - Q_2$ with disc$(Q_1) = 1$ for a suitable rational point of ${\mathbb P}^1$ at 
infinity.   

Since $C$ has a   zero-cycle of degree 1,    
by (\cite[p. 599]{CT2005}),  $C$ has a zero cycle $\sum n_ix_i$  
of degree one  with $x_i$ closed points of $C$ which are unramified over ${\mathbb P}^1$. 
Let $L_i = \kappa(x_i)$ be the residue field at $x_i$. 
Then there exists $P_i \in  C(L_i)$ which is unramified over ${\mathbb P}^1_{L_i}$. 
By (\ref{lgp}), $X(L_i) \neq \emptyset$ for all $i$.  Since $\sum n_i [L_i : k] = 1$, $[L_j : k] $ is odd for some 
$j$.  Since $X(L_j) \neq \emptyset$, the quadratic form $tQ_1 - Q_2$ is isotropic over $L_j(t)$ and by Springer's theorem, 
$tQ_1 - Q_2$ is isotropic over $k(t)$. Using  Amer Brumer theorem,  one concludes that $X(k) \neq \emptyset$. 
\end{proof}
  
 \appendix

 \section{\bf A note on hyperelliptic curves and pencils of quadrics}
 \smallskip
 \begin{center}   \textsc{S. Ramanan} \end{center}
 \label{Appendix}
 
\subsection{Introduction}~

There is a deep relationship  between a hyper-elliptic curve of genus $g > 0$,
and a pencil of quadrics in $\C ^{2g +2}$. This goes back to Gauthier  (\cite{LG}).
The space of $(g+1)$-dimensional (maximal) isotropic subspaces for
a non-singular quadric has two connected components. If the quadric 
has rank 1 less, maximal isotropic subspaces are still $(g+1)$-dimensional but
the space of such subspaces is connected. Thus pairs consisting 
of a quadric of the pencil (assuming it is generic) and a component of the space of maximal
isotropic subspaces in that quadric, constitute a 2-sheeted covering over $\P ^1$ or what is the
same, a
hyper-elliptic curve. Hence  it is not surprising that many constructions 
over the hyper-elliptic curve have corresponding geometric interpretations 
 in terms of the pencil. 

In particular, initially the Jacobian of the curve was described in terms of the
variety of $g$-dimensional isotropic spaces for all quadrics of the pencil.
Narasimhan and I in the case $g =2$, and Desale and I  (\cite{De-Ra}) in the general case, gave a similar 
description for the moduli varieties of vector bundles of rank $2$ over the curve, in terms of the geometry 
of the pencil.  Still later, I gave such a meaning  (\cite{Ra1})  to the variety of isotropic subspaces of other dimensions 
in terms of suitable moduli on the curve. 

More recently, Manjul Bhargava \cite{Manjul})  considered similar properties over a number field.
I am giving below a proof of the results mentioned above in such a way that they
are valid for curves over any field of characteristic $\neq 2$.   
This requires an assumption that the curve itself admits  a rational point. 
I have given full details of the construction in proving the isomorphism of the Jacobian with 
the variety of $g$-dimensional isotropic subspaces for the pencil and have sketched how 
similar considerations in the case of vector bundles of rank 2 are also valid over such fields. 

Parimala Raman  and Jaya Iyer  have been studying Selmer groups and Hasse
principle etc. and wished for such a proof. I am thankful to Parimala for 
encouraging me to re-visit what I had done half a century ago.  
  \\
 
 \subsection{Preliminaries and Notation}~

All considerations below are over any field $k$ of characteristic $\neq 2$.

Let $U$ be a 2-dimensional vector space and $P= P(U)$ the corresponding
projective line. A general element $s$ (with all roots nonzero and distinct over
the algebraic closure) of $\Gamma (P , {\OO}(2g +2)) = S^{2g +2}(U^*), g\geq 2$ 
defines an hyper-elliptic curve $C$ of genus $g$, namely 
$Spec(\OO\oplus \OO(-g-1))$ over $P$, with the obvious algebra structure on this 
$\OO_P$-Module, which uses the map 
$Sym ^2\OO(-g-1) = \OO(-2g -2) \to \OO$ 
given rise to by $s$. This realises $C$ as a $2$-sheeted covering $\pi :C \to P$.

The automorphism of the sheaf of algebras $Spec(\OO\oplus \OO(-g-1))$ 
acting as $+1$ on $\OO$ and $-1$ on $\OO(-g-1)$ gives rise to an involution 
$\iota $ of $C$ and the quotient is $P$. The $\iota $-fixed divisor $W$ (of degree $2g +2$)
in $C$ maps isomorphically  on the ramification divisor $D$ (which is the divisor in
$P$ defined by $s$). The hyperplane bundle $\OO(1)$ on $P$ has as inverse image, 
a line bundle $h$ on $C$ of degree $2$.  The divisor class of $D$ (resp. $W$) is 
$\OO(2g + 2)$ (resp. $h^{g+1}$). The section $s$ lifts to a section of $h^{2g+2}$ and 
is the square of a section of $h^{g +1}$ (unique up  to $\pm 1$) with $W$ as the divisor 
defined by it. 
 
For any line bundle $L$ on $C$ of degree $d$, we have an isomorphism (unique up to a scalar
multiple) $L\otimes {\iota ^*(L)} \simeq h^d$.
In particular, any line bundle $\alpha $ of degree $0$ 
gives rise to a trivialisation of $\alpha \otimes \iota ^*(\alpha  ^{-1})$ unique up to $\pm 1$.

In what follows we will adopt the following notation. For any vector bundle  $E$ on $C$ (over $k$ or 
any extension of $k$), we denote the vector space $\Gamma (W, E|W )$ by $W(E)$  and the image 
under the restriction map of $\Gamma (C,E)$  to  $W(E)$, by $C(E)$.

We denote the top exterior power of any (finite dimensional) vector space (or vector bundle)
$E$ by $det(E)$. We assume that the $1$-dimensional vector space $det(U)$ is identified 
with the field.  Consequently we may identify $U$ and $U^*$.
 \\

 \subsection{ Duality in $W$}~

The sheaf  $\OO(-2)$ is naturally the dualising sheaf $K_P$ (or what is the same, 
canonical bundle) of $P$.  The dualising sheaf $K_C$  of $C$ is consequently 
 $\pi ^*(K_P) \otimes \OO(W) \simeq h^{-2} \otimes h^{g+1} = h^{g-1}$. 
Also, the dualising sheaf $K_W$ of the divisor $W$ is the tensor product of the restriction to $W$ of 
$K_C$ and $\OO(W)$, namely $(h^{g-1}\otimes h^{g+1})|W = h^{2g}|W$. 

The duality in $C$ and that in $W$ are compatible in the following sense. If $E $ is any vector bundle on $C$
of degree $d$, we have  $\Gamma (C, E)$ is dual to $H^1(C, h^{g-1}\otimes E^*) $ on the one hand and  $W(E )$ 
is dual to $W(h^{2g} \otimes E^*) $ on the other. The transpose of the restriction map $\Gamma (C, E) \to W(E)$ is 
given by the connecting homomorphism of the exact sequence 
$$0\to  h^{g-1} \otimes E^* \to h^{2g} \otimes E^*  \to (h^{2g} \otimes E^*)|W  \to 0$$ namely,  $W(E)^* = W(h^{2g} \otimes E^*) 
\to H^1(C, h^{g-1}\otimes E^*) = \Gamma (C, E)^*$. The duality in $W$ can also be described by the map 
$W(E)\times W(h^{2g}\otimes E^*) \to W(h^{2g})$, followed by the connecting homomorphism of the  exact sequence
$$0 \to h^{g-1} \to h^{2g} \to h^{2g}|W\to 0.$$

Also, if $f:E\to F$ is any homomorphism, it gives a map $C(f):\Gamma (C, E) \to \Gamma (C, F)$.  One knows that the 
transpose of $C(f)$ is the induced homomorphism $H^1(C, h^{g-1}\otimes F^*) \to H^1(C, h^{g-1}\otimes F^*)$. 
Similarly we have a map $W(f):W(E)\to W(F)$ whose transpose is given by the natural map $W(h^{2g}\otimes F^*)
\to W(h^{2g}\otimes E^*)$.
\\

\subsection{Properties of a generic pencil}~

We wish to study properties of generic pencils of quadrics.  
We recall first some well-known properties of a quadric defined by a quadratic form $Q$. If $V$ is a vector space of dimension $n$, 
and $Q$ is a quadratic form on $V$, often we will  denote the corresponding quadric in $P(V)$  also by $Q$. We can also think of 
$Q$ as a symmetric linear map $V\to V^*$. The kernel of the above map is called the {\it nul-space} of $Q$. 
A subspace $A$ of $V$ is {\it isotropic} for $Q$ if the restriction of the quadratic form to $A$ is zero. If $n = 2g +2$ 
and $Q$ is non-degenerate or has nul-space of dimension 1, then the maximal isotropic subspaces have dimension $g+1$. 
The subvariety of the Grassmannian of $(g+1)$-dimensional subspaces of $V$, consisting of isotropic subspaces, 
has two irreducible components if $Q$ is non-degenerate and is itself irreducible if the nullspace of $Q$ is of dimension 1. 
Orthogonal transformations of $Q$ act on the space of isotropic subspaces of fixed dimension $d$. If $Q$ is non-degenerate, 
and $d = g+1$, an orthogonal transformation preserves each component if and only if its determinant is 1. 

Let $P = P(U)$ be a projective line and $V \to V^*\otimes \OO(1)$ be a family of quadrics parametrised by $P$. 
We call it a {\it pencil} of quadrics in $V$. We identify $det(V)$ with the field $k$. 
Its determinant $ \delta $ (called the {\it discriminant}) is a section of $\OO(n)$. In the following, we always assume 
that $\delta $ is non-zero, which is the same as saying that the generic quadric of the pencil is non-degenerate.
 If $\delta $ has $n$ zeros (in the algebraic closure) with multiplicity $1$ each, we call it a {\it generic pencil}. 
 The divisor in $P$ defined by $\delta $ will be denoted by $D$.
 Those quadrics which correspond to the zeros  of $\delta  $ are then singular quadrics with 
 nullspaces $N_i$ of dimension 1 each, and all other quadrics are smooth. 

If $Q$ and $Q'$ are two singular quadrics in a generic pencil with nullspaces $N$ and $N'$ respectively, 
then $N$ and $N'$ are orthogonal for both $Q$ and $Q'$ and hence for all quadrics of the pencil. 
Thus $V$ is the direct sum of  the nullspaces $N_i$ of singular quadrics of 
the pencil and therefore has a basis which is orthogonal for all the quadrics of the pencil. 

We will now recall some properties of a generic pencil over an algebraically 
closed field of characteristic $0$ in a vector space $V$ of even dimension,  say $2g +2$.  \\

  \subsubsection{ Dual Family of quadrics} ~

 \begin{lemma}
 Assume given a generic pencil on $V$, and that $H$ is any hyperplane in $V$.
 Then the generic quadric in the induced pencil on $H$  is non-degenerate.
 \end{lemma}
 
 \begin{proof}
 Since $H$ cannot contain all the nullspaces of the pencil, let us assume that some nullspace $N$ is not a subspace of $H$. 
 The corresponding singular quadric factors through a non-degenerate quadric in $V/N$. 
 Then $H$, along with the restriction of this quadric to $H$, maps isomorphically on the quadratic space $V/N$. 
 Hence it, and hence the generic quadric in $H$ of the induced pencil, are non-degenerate
 \end{proof}

\begin{remark}
In Equivalently, the restriction of the pencil to any hyperplane $H$ has as its discriminant 
a non-zero section of $\OO(2g +1)$. On identifying $V^*$ with $\Lambda ^{2g +1}(V)$, 
the $(2g +1)$-th exterior map $\Lambda ^{2g + 1}(V) \to \Lambda ^{2g +1 }(V)^*\otimes \OO(2g +1)$
gives rise to a quadratic map $V^* \to \Gamma (P, \OO(2g +1))$. In other words, this
 gives a dual family of quadrics, parametrised
by $\Gamma (P, \OO(2g +1))^*$.   The restriction map $\Gamma(P, \OO(2g +1)) \to \Gamma (D, \OO(2g +1))$ 
is an isomorphism 
since both have dimension $2g +2$ and the kernel is a section of $\OO(-1) = 0$
Thus $\Gamma (P, \OO(2g +1))^*$ can be identified with the direct sum of the fibres at $\pi w$ 
of the line bundle  $\OO(2g +1)$ on evaluation. Hence  we have  a family of diagonal 
quadrics with respect to the dual decomposition $V^*\simeq \Sigma N_w^*$. 
The space $\otimes ^2(N_w)^*$ is mapped as above isomorphically on the 
1-dimensional subspace  of sections of $\OO(2g +1)$ which vanish at all points of 
$W$ other than $w$, or what is the same, the fibre at $w$ of $\OO(2g +1)$. 
If $s \in V^*$ the corresponding section of $\OO(2g +1)$ evaluated at $w$, 
is simply the image of $s | N_w$ under the above map.In other words, 
we have a canonical quadratic map $V^* \to \OO(2g +1))|{\pi w}$ for every $w\in W$.
\end{remark}   

\vskip 5mm

  \subsubsection{ Isotropic subspaces for the pencil} ~

 A subspace of $V$ is said to be {\it isotropic for the pencil} if it is isotropic  for all quadrics of the pencil.
The variety of $r$-dimensional subspaces of $V$ isotropic to the pencil, will be denoted $I_r$. 
As stated above, all quadrics of the pencil are linear combinations of any two quadrics $Q_1$
 and $Q_2$ (both of which we may assume non-degenerate). Obviously, 
 a  subspace is isotropic for the pencil  if and only if it is so with respect to 
 $Q_1$ and $Q_2$. With reference to a basis $S = \{e_i\}$, as mentioned above, 
 he  two quadrics may be put in the form $\Sigma x_i^2 = 0$ and $\Sigma a_ix_i^2 = 0$,
  where all the $a_i$ are distinct and non-zero. 

The space consisting of pairs $(Q, c)$ where $Q$ is the quadric of the pencil and $c$ is a 
connected component of the space of maximal isotropic spaces for $Q$, is a 2-sheeted covering of $P$  
ramified over $D$ and may be  identified with an hyper-elliptic curve $C$ of genus $g$. 
The set $W$ in $C$ corresponds to the set of singular quadrics, and hence to the set of their nullspaces, 
and we may as well write $V = \sum N_w, w\in W$. Moreover, if $V' = \sum N_w^2$, the quadratic forms 
of the pencil are obtained as linear forms on $V'$ composed with the natural quadratic map $V\to V'$.   

There does not exist any $(g+1)$-dimensional subspace isotropic for the pencil. Such a subspace, 
being a maximal isotropic space for every singular quadric of the pencil, has as its image modulo the 
nullspace of the quadric, an isotropic subspace of a non-degenerate quadric in a $(2g + 1)$-dimensional 
space and hence has dimension $\leq g$. In other words, it contains the nullspace of every singular quadric.  
Thus it has to contain all the basis elements, which is absurd.\\

 \subsubsection{ Automorphisms of the pencil}~
 
 The group $G$ of automorphisms which leave both the quadratic forms invariant,
  modulo $\pm Id$ is $(\Z/2)^{2g +1}$, namely those which act on the spaces $N_w$ by $\pm 1$.

Let $W'$ be a proper subset of the basis set $W$, of even cardinality $2p$, and $\sigma $ 
an involution of $W(\eta )$ which acts as $+Id$ on  all the $N_i, i\in W'$ and as $-Id$ on 
all $N_j, j\in W\setminus W'$. Then $\sigma $ acts on the set of $g$-dimensional isotropic 
subspaces for the pencil, without fixed points. For, if $L$ is a subspace isotropic for the pencil 
and invariant under $\sigma $, then $L$ is the direct sum of $L^+$ and $L^-$, namely the intersection 
of $L$ with the eigenspaces of $\sigma $. Since $L^+$ and $L^-$ are isotropic subspaces for 
generic pencils in spaces of dimensions $2p$ and $2g + 2- 2p$ respectively, we have 
$dim L^+ \leq p-1$ and $dim L^- \leq g - p$ proving that $dim L \leq g-1$.  \\

 \subsubsection{ Morphism $I_g \to J^{2g +1}$} ~

 If $A$ is a $g$-dimensional subspace of $V$ isotropic for the pencil, then for every $c\in C$, 
 we have a unique $(g+1)$-dimensional subspace $E_c$, which is  isotropic for the quadric $Q$ defined 
 by $\pi c \in P$, contains  $A$ and belongs to the component $(Q,c)$. This defines a morphism of $C$ into 
 $P(V/A)$. The image of $C$ is not contained in any hyperplane of $V/A$, namely, any hyperplane $H$ in 
 $V$ containing $A$. For, the restriction of the pencil to the hyperplane $H$  in $V$ has smooth quadrics in 
 general, and equivalently the restriction $H\to H^*\otimes \OO(1)$ has a nonzero section of $\OO(2g +1)$
  as its discriminant as we have seen above. Thus $c\in C$ is a point of this hyperplane if and only if the quadric $Q_c$ 
  corresponding to $c$ intersects $H$ in a quadric which contains a $g+1$-dimensional isotropic subspace containing $A$, 
  namely the inverse image of $\{c\}$ under the map $H \to H/A$. This is possible 
if and only if $Q_c \cap H$  is singular. Hence the mapping $C\to P(V/A)$ pulls back the line bundle $\OO(1)$ 
to a line bundle $\zeta $ on $C$  of degree $2g + 1$. The induced linear map $(V/A)^*\to \Gamma (C, \zeta )$
 is injective and hence (both being of dimension $g + 2$), an isomorphism. Thus to any $g$-dimensional 
 subspace $V$ isotropic for the pencil, we have associated a line bundle $\zeta  $ of degree $2g + 1$ on $C$ 
 with an isomorphism of $\Gamma (C, \zeta )$ with $(V/A)^*$. We denote this morphism $I_g \to J^{2g + 1}$ by $f$. 

Since the inverse image of $w\in C$ in $V$ under the map $V \to V/A$ is the (direct) sum of $N_w$ and $A$,
 it follows that $N_w$ maps isomorphically to $(\zeta ^*)_w$ under the natural map $V \to V/A$.  
 These isomorphisms determine the surjection $V\to V/A$.

We have seen above that there is a natural quadratic map $q: V^* \to \Gamma (P, {\OO}(2g +1))$. 
This takes non-zero vectors to non-zero 
vectors and so gives a morphism of the corresponding projective spaces. Indeed, it is clear that it induces 
an isomorphism of the quotient of $P(V)$ by the finite group $(\Z/2)^{2g +2}$ which acts as $\pm 1$ on $N_i$, 
with $P\Gamma (P, \OO(2g +1))$.

\begin{remark} We have essentially shown that there is a natural isomorphism $\sum
 N_w^2\simeq  (\Gamma(P, \OO(2g +1)))^*$. Besides, it is easy to check that the
  pencil is given by the linear map $U = \Gamma (P, \OO(1))^* \to \Gamma (P, \OO(2g +1))$ given rise to, by $s$.
  \end{remark}

The map $q$ composed with the inclusion $\Gamma (C, \zeta) = (V/A)^* \to V^*$, yields a quadratic map 
$\Gamma (C, \zeta ) \to \Gamma (P, \OO(2g +1))$. On the other hand, we have also another natural 
quadratic map $\Gamma (C, \zeta ) \to \Gamma (P, \OO(2g +1))$ given by $\sigma \mapsto \sigma \otimes
\iota^*(\sigma)$, noting that this image is an $\iota $-invariant section of $h^{2g + 1}$, and hence gives rise to a 
section $\pi s$ of $\OO(2g + 1)$.  Both of these quadratic maps take non-zero vectors into non-zero vectors.
The morphism $q$ maps any $s \in V^*\neq 0$, to a section 0f $\OO(2g + 1)$ whose zeros correspond to 
the singular quadrics of the pencil restricted to the hyperplane $s=0$.  If $s$ belongs to $\Gamma (C, \zeta )$, 
then $qs$ is the same as the section of $\OO(2g + 1)$, the divisor of whose zeros is the image in $P$ of the
 divisor of zeros of $s$ in $C$. We have just seen that these correspond to the singular quadrics in the 
 hyperplane $s=0$ in $V$. Hence $q$ restricted to $\Gamma (C, \zeta )$ is the map $s\to \pi (s)$, up to a scalar. 
 Modifying the map $V \to \Gamma (C, \zeta )^*$ by a scalar, we may assume that the restriction of $q$ to 
 $\Gamma (C, \zeta )$ is $\pi $. In particular, the isomorphism  $N_w \to \zeta _w^*$  induces an isomorphism 
 of $N_w^2$ with $\zeta _w^{-2}$, which depends only on $\zeta $.

Let $A_1$ and $A_2$ be two elements of $I_g$, with $f(A_1) = f(A_2) = \zeta $. 
We have then two surjections of $V$ onto $(\Gamma (\zeta ))^*$  or transposing, two inclusions 
$\Gamma (C, \zeta)$ in $V^*$, whose compositions with $q$ are the same. 
This implies that we have an automorphism $\varphi $ of $V$ which act as $\pm 1$ on 
$N_w$'s taking $A_1$ to $A_2$. Moreover, its determinant is $1$ since it preserves the 
component corresponding to $c$ on any non-degenerate quadric of the pencil. 

\begin{prop}
Any two elements of $I_g$ with the same image under $f$ are taken to one another by a transformation of 
$V$ which is $Id$ or $-Id$ on an even number of $N_i$'s.
\end{prop}

\vskip 3mm

 \subsection{Pencil of quadrics on $\Gamma (W, \eta )$}~
\label{pencil-eta}

Let $C$ be defined over any field $k$ of characteristic $\neq 2$. We will assume that $C$
 has a rational point $x$. Hence we have a 
$k$-rational line bundle $\eta $ of degree $2g -1$, for example 
$h^g \otimes \OO(-x)$. Let $\xi = h \otimes \iota ^*(\eta )$.

We denote $\Gamma (W, h^{2g - 1})$ by $V'$. There is a natural tensor product map 
$V' \times \Gamma (C, h) \to \Gamma (W, h^{2g})$. 
Composing this with the boundary homomorphism $\Gamma (W, h^{2g}) \to H^1(C, K_C) = k$, 
we get a homomorphism $\Gamma (C, h) \to (V')^*$. The pencil we construct is as follows. 
We will give a $2g +2$-dimensional $k$-vector space $V$ with a quadratic map $V\to V'$. 
Composing this with the  the above family of linear forms on $V'$, parametrised by $\Gamma (C, h)$,  
we get a pencil of quadratic forms on $V$. If $\eta  $ is a rational line bundle of degree $2g -1$ as above, 
we take $V$ to be $\Gamma (W, \eta ) \simeq \Gamma (W, \xi )^*$ and the 
quadratic map $V \to V'$ to be $s\mapsto s \otimes \iota ^*(s)$.    

By going over to the algebraic closure, one sees that the discriminant of the above 
pencil vanishes only for those sections of $h$ which vanish
at a point of $W$ with multiplicity $1$. Thus we have a generic pencil 
whose discriminant is $s$ (up to a non-zero scalar).  

We now have the rather obvious 
 
\begin{remark}~

\begin{enumerate}
\item[1)]{If the degree $d$ of any line bundle $\zeta $ on $C$ is less than
$2g + 2$, the restriction map $\Gamma (C, \zeta )\to W(\zeta )$ is injective, since its kernel is
 $\Gamma (C, \zeta \otimes {\OO}(-W)) = \Gamma (C, \zeta \otimes h^{-g-1}) = 0$, the degree 
  of $\zeta \otimes h^{-g-1}= d-2g -2$ being negative by assumption.} \\
  
\item[2)]{Also, if  $d$ is greater than $2g - 2$, then $H^1(C, \zeta ) = 0 $. 
Hence if $2g-2< deg(\zeta) < 2g +2$, we may associate to $\zeta $,
 a $(d+1-g)$-dimensional subspace $C(\zeta )$ of $W(\zeta )$, 
 namely the image of $\Gamma (C, \zeta )$ in $\Gamma (W, \zeta )$. 
 Moreover, if $d > 2g -1 $, then $\zeta $ is the pull-back of the bundle 
 ${\OO}(1)$ by the natural map of $C$ into $P\Gamma (C, \zeta )^* = PC(\zeta )^*$.}
 \end{enumerate}
\end{remark}

\begin{remark}~

\begin{enumerate}
\item[1)] {The assumption that there is a rational point in $C$ implies that  all the Jacobians $J^d$ for all $d$,
 may be identified using tensorisation by $\eta $ and/or  a suitable power of $h$.}\\
 
\item[2)] {There is a canonical Poincar{\'e} bundle on $J^g$, namely the inverse image of the $\Theta $ divisor
in $J^{g-1}$ by the map $(\zeta , x) \to \zeta \otimes {\OO}(-x)$ of $J^g \times C \to J^{g-1}$. 
Using the above remark,  we may assume that there is a Poincar{\'e} bundle in any degree $d$.} 
 \end{enumerate}
 \end{remark}

The quadratic map $W(\eta ) \to W(h^{2g-1})$  when restricted to $C(\eta )$ factors through 
$\Gamma (C, h^{2g-1})$. Hence on tensoring with
$\Gamma (h)$ the restriction of the pencil to $W(\eta )$ factors through 
$\Gamma (C, h^{2g}) \to  \Gamma (W, h^{2g}) $ and hence goes to 0 
under the boundary homomorphism $W(h^{2g}) \to H^1(C, h^{g-1})$ by exactness. 

Thus we have associated to each line bundle $\eta $ of degree $2g -1 $  over $k$,
\begin{enumerate}
\item[1)]  {a pencil of quadrics on $W(\eta )$ parametrised by $P$, and}  \\

\item[2)] {a $k$-subspace $C(\eta )$ of $W(\eta )$ of dimension $g$ to which the restrictions of all quadrics of the pencil are 0.} \\

\item[3)] {To  any rational point $c\in C$ corresponds  a maximal isotropic $k$-subspace for the quadric determined by a section of $h$
which vanishes at $c$ and $\iota c$ and contains the above subspace $C(\eta )$.
  }
\end{enumerate}

\vskip 5mm

\subsection{ Jacobian as the space of isotropic subspaces of the pencil}~

Let $\eta $ be a $k$-rational line bundle of degree $2g -1$ and 
$\xi = h\otimes \iota ^*(\eta )$.  We will denote by $J$, the Jacobian of any degree and by $P$, a Poincar{\'e} 
bundle on $J\times C$.
Let $Q$ be the line bundle $P\otimes (1\times \iota )^*(P)^{-1}$ on $J \times C$. Clearly the restriction of 
$Q$ to $J\times W$ is trivial. The direct image $E$ on $J$ of $Q\otimes p_C^*(\eta )$ (resp. $F$ of $Q\otimes p_C^*(\xi )$)
 is a vector bundle of rank $g$ (resp. $g + 2)$, while the direct image on $J$ of the restriction of $Q\otimes p_C^*(\eta )$ to 
 $J\times W$, is a trivial vector bundle of dimension $2g +2$ with fibre $\Gamma (W, \eta )$. We have now (denoting 
 $\alpha \otimes \iota ^*(\alpha )^{-1}\otimes \eta $ by $L(\alpha )$, $\alpha \in J$) the exact sequence
$$ 0\to \Gamma (C, L(\alpha )) \to \Gamma (W, \eta ) \to H^1( C, h^{-g -1}\otimes L(\alpha )) \to 0.$$
Since $L(\alpha ) \otimes \iota ^*(L(\alpha ))$ is canonically isomorphic to $h^{2g -1}$, we have the exact sequence 
$$0\to E \to \Gamma(W, \eta ) \otimes {\OO} \to F^* \to 0$$
on $J$ with $E_{\alpha } = \Gamma(C, L(\alpha ))$ and $F_{\alpha } = \Gamma (C, h\otimes \iota ^* L({\alpha }))$.
 
In particular, this defines a morphism $S$ of $J$ into the $g$-dimensional Grassmannian of $W(\eta )$. 
Moreover, the group $J_2$ consisting of of line bundles $\alpha $ with $\alpha ^2 $ trivial, acts on $J$. Since, for each such
element we have an isomorphism $\alpha \simeq \iota ^*\alpha $ (which, we may assume has Id as its square)
  it associates $\pm 1$ for every $w\in W$ (in the algebraic closure). This gives an isomorphism of $J_2$ with 
the group of automorphisms of the pencil modulo $\pm Id$. 

\begin{lemma}The map $S: \alpha \mapsto \Gamma (W, L(\alpha )) \subset
W(\eta )$ of $J$ into the Grassmannian of $g$-dimensional subspaces of $W(\eta )$ is compatible with the action of $J_2$ on $J$ and the action of the automorphism group of the pencil on the Grassmannian.
\end{lemma}

\begin{proof} In fact, it is a consequence of the isomorphism of $J_2$ with the automorphism group, which uses the fact that
each $\alpha \in J_2$ corresponds to an automorphism of $W(\eta )$ which acts by $\pm 1$ on the basis of $W(\eta )$.
\end{proof}

We claim that this  map $S$ is actually into the space $I_g$ of $g$-dimensional subspaces which are isotropic for all
the quadrics of the pencil defined above. In order to show this, for simplicity we will extend the above consideration 
to the algebraic closure of $k$ and prove it there. 

Now we restate our construction in this set-up. If $\alpha $ is a line bundle, associate to it the line bundle 
$L(\alpha ) =  \alpha \otimes \iota ^*(\alpha)^{-1}\otimes \eta $ on $C$. Then $L(\alpha ) = L(\beta)$ if and only if $j =
\alpha ^{-1}\otimes \beta $
satisfies $j =\iota ^*(j)$, or what is the same, $j\in J_2$.  In other words, $L:J \to J^{2g -1}$ is a Galois covering with $J_2$ 
as Galois group. 
Consider then the restriction  $\Gamma (C, L(\alpha )) \to W(L(\alpha ))$, which is injective by 5.1. 
Since $\iota $ acts as Id on $W$, it follows that $W(L(\alpha ))$ is the same as $W(\eta )$.  
Hence $C(L(\alpha ))$ is a $g$-dimensional subspace $S(\alpha )$ of $V = W(\eta )$. 

Even if $L(\alpha )$ is isomorphic to $L(\beta )$, it does not imply that the two subspaces $S(\alpha )$ 
and $S(\beta )$ are the same, since
the isomorphism of  $L(\alpha )$  with $L(\beta )$  need not restrict to identity on $W$. 
 This restriction is the same as the restriction of the isomorphism $j\to \iota(j)$ to $W$. 
 Assuming the isomorphism has square $Id$,  it is unique up to $\pm 1$. It acts on $W$ as $+1$ on 
a subset of $W$ of even cardinality and as $-1$ on the complement. See [3, Lemma 2.1].
Thus we obtain an isomorphism of $J_2$ with a subgroup of automorphisms of $W(\eta )$
respecting the pencil of quadrics, modulo $\pm Id$.
Consequently, $J_2$ acts on $I_d$, for all $d$, in particular on $I_g$. Moreover,  
the morphism $S$ commutes with the action of $J_2$ on $J$ and  that on $I_g$. 
 
We have the tensor product  map
$$\Gamma (C, h)\otimes \Gamma (C, L(\alpha ))\otimes \Gamma (C,\iota ^*L(\alpha )) \to \Gamma(C, h^{2g})$$
on the one hand, and 
$$\Gamma (C, h)\otimes W(\eta )\otimes W(\eta ) \to W(h^{2g})$$
on the other. Note that the induced pencil on $S(\alpha )$ is given by 
$\Gamma (C, h) \otimes S(\alpha ) \otimes S(\alpha ) \to k$. It is therefore clear that this factors through $C(h^{2g})$. 
Thanks to  the exact sequence mentioned above, this is therefore 0 and hence 
$S(\alpha )$ is isotropic for all quadrics of the pencil. 
Thus we have associated to each $\alpha \in J$, a $g$-dimensional subspace $S(\alpha )$ of $W(\eta )$
 isotropic for the pencil and thus given a morphism of $J^g$ into $I_g$.  Hence we have the

\begin{prop}
\label{the-maps-s} The map $\alpha  \to S(\alpha )$ is a $k$-rational morphism of $J$ into $I_g$ which respects the action of $J_2$.
\end{prop}

 \begin{theorem}
  The morphism $S$ of $J^g$ into $I_g$ is an isomorphism.
  \end{theorem}
  \begin{proof}
We have given the map over $k$ assuming the existence of a rational line bundle of degree $2g -1 $.
 It is enough to show that it is an isomorphism over the algebraic closure. We have already given a morphism $f$ of $I_g$ into $J^{2g +1 }$.
In the following we will show that the composite $f \circ S: J\to J^{2g +1}$ is the map $\alpha \mapsto h\otimes \iota ^*L(\alpha )$.
\end{proof}

In the following $\eta $ is {\it any} line bundle of degree $2g -1$ and $\xi =h\otimes \iota^*(\eta )$. 

\begin{lemma}
The space $C(\xi )$ is the orthogonal complement of $C(\eta )$ under the duality.
\end{lemma}

\begin{proof} In fact, the duality restricted to $C(\eta ) \otimes C(\xi )$ is given by the natural map   
$\Gamma (C, \eta )\otimes \Gamma (C, \xi ) \to \Gamma (C, h^{2g})$, then restricting it to $W$ and finally composing 
it with the connecting homomorphism $W(h^{2g}) \to H^1(C, h^{g-1})$ of the exact sequence 
$$0\to h^{g-1}\to h^{2g} \to h^{2g}|W \to 0. $$
This composite is clearly $0$ in view of the exactness of the associated cohomology exact sequence. 
Thus $C(\xi )$ is contained in the orthogonal complement of $C(\eta )$. 
Since $C(\eta )$ is of dimension $g+2$, our assertion is proved.
\end{proof}

Thus we have the exact sequence
$$0 \to C(\xi ) \to W(\xi ) \to C(\eta )^*\to 0$$
as well as the dual exact sequence 
$$0 \to C(\eta ) \to W(\eta ) \to C(\xi )^*\to 0.$$

Using the natural imbedding of $C$ in $P\Gamma (C, \xi )^*$, we get a morphism of $C$ into the subvariety of the Grassmannian 
of $(g +1)$-dimensional subspaces in $W(\eta ) $ consisting of subspaces that contain $C(\eta )$.  

\begin{lemma}
For every $x\in C$, the corresponding element $G_x $ of the Grassmannian has as its orthogonal complement in $W(\xi )$
the subspace of $C(\xi )$ given by the image of the subspace of $\Gamma (C, \xi )$ which vanish at $x$.
\end{lemma}

\begin{proof}
This follows from the duality of the above two exact sequences.
\end{proof}

Given $x\in C$, we have the quadric $Q_x$ corresponding to the point in $P\Gamma (C, h) $ which vanishes at $x$. 
We will assume first that $x\notin W$.
Then $W(\eta ) \to W(\eta \otimes {\OO}(x))$ is an isomorphism.
 It is clear that $G_x$ is the same as $C(\eta \otimes {\OO}(x))$ imbedded in $W(\eta \otimes {\OO}(x)) = W(\eta )$. 
It gives rise to the map given as follows, as remarked above.  Consider the tensor product map  
$\Gamma (\eta \otimes {\OO}(x)) \times \Gamma (\iota ^*(\eta )\otimes {\OO}(\iota x))\to \Gamma (h^{2g})$ 
and compose this with $\Gamma (h^{2g}) \to W(h^{2g} )\to H^1(h^{g-1}) = k$. This is clearly $0$.
 Now $\iota ^*(\eta )\otimes {\OO}(\iota x) \simeq \xi \otimes h^{-1}\otimes {\OO}(\iota x)\simeq \xi \otimes {\OO}(-x)$. 
 The subspace $\Gamma (\eta \otimes {\OO}(x))$ is thus isomorphic to the space $H_x$ of sections of $\xi $
  vanishing at $x$. We conclude that with respect to the quadric corresponding to $x$, the spaces $G_x$  
  is a maximal isotropic subspace for the quadric $Q_x$ and contains $C(\eta )$ which is isotropic to all quadrics of the
pencil. 

This also gives every $x\in C$, one of the (two) components of the variety of maximal isotropic subspaces in the 
quadric $Q_x$.

\begin{remark}
If $x$ belongs to $W$,  $Q_x$ is of rank $2g +1$ and it is easy to check that $G_x$ is in this case the $(g+1)$-dimensional 
(that is, maximal) subspace contained in $Q_x$ and containing $C(\eta )$. It is the only such space in this case.
\end{remark}

To sum up, we have the diagram:\\
$$
\begin{array}{ccccccccc} 
&&&& 0&  & 0\cr
 &&&&\downarrow&&\downarrow  \cr
 & & & &N_x & \to & \xi _x & & \cr
               & & & & \downarrow& &\downarrow &&\cr
0&\to & C(\eta )  & \to & W(\eta ) & \to &  \Gamma (\xi )^* & \to & 0 \cr
                 & &\downarrow  & & \downarrow&& \downarrow &&\cr
0&\to & C(\eta \otimes {\OO}(x)) & \to & W(\eta \otimes {\OO}(x)) & \to & \Gamma (\xi \otimes {\OO}(-x))^* & \to & 0
\end{array} 
$$         
\vspace{5mm}

If $x\notin W, N_x = 0$. If $x\in W, N_x$ is 1-dimensional. In any case, the inverse image of $\xi _x$ in $W(\eta )$, namely $G_x$ is 
the same as the inverse image of $C(\eta \otimes {\OO}(x))$ by the map $W(\eta )\to 
W(\eta \otimes {\OO}(x))$. Hence the imbedding of $C$ in the linear system of $\xi $ is 
determined by $\{G_x, x\in C\}$. But then $G_x$ is the space determined by $x$ 
from among the two subspaces which are isotropic for $Q_x$, if $x\notin W$, and is the unique
maximal isotropic space in $Q_x$ if $x\in W$.  This implies that the space $\Gamma (\xi )^*$ as well as the natural imbedding of 
$C$ in its projective space are determined by $S(\eta )$. It is clear that we may replace $\eta $ in this argument  by any line bundle
$L(\alpha )$. 

This leads to the
\begin{prop} 
For every $\alpha \in J$, the image under $f$ of $S(\alpha )$ is $h\otimes \iota ^*(L(\alpha ))$.
\end{prop}

This essentially proves the Main theorem. Firstly, the map $S: J^{2g -1}\to I_g$ is surjective, for given any $B\in I_g$, let $\alpha $
be any element of $J^{2g - 1}$ with $f(B) = h\otimes \iota ^*( L(\alpha ))$. 
Then $S(\alpha )$ and and $B$ have the same image under $f$. 
We have shown that in that case, $S(\alpha )$ and $B$ are taken
 one to another by an element of $J_2$. Thus modifying $\alpha $
 by this element of $J_2$, we see that $B$ is in the image of $S$. 
 In particular, this shows that the morphism $J^g \to I_g$ is onto, 
 and hence $I_g$ is irreducible. Finally, since the quotient under the
  free action of  $J_2$ on both $J^g$ and $I_g$ is $J^{2g -1}$, the map $S: J^g \to I_g$ is an isomorphism. \\

\subsection{Moduli of rank 2 vector bundles and the space of isotropic subspaces of the pencil}~

The interpretation of the moduli space of vector bundles of rank $2$ and determinant $\eta $, is similar. 
If $E$ is a stable vector bundle on $C$ of rank 2 and determinant $\eta $,
 then $E\otimes \iota ^*(E)$ is semistable with slope $>2g -2$ and hence  its first cohomology vanishes. Also the natural restriction map 
$r$ of $\Gamma (C, E\otimes \iota ^*E)$  (of dimension $8g-4+4(1-g) = 4g$) to $W(E\otimes E)$ (of dimension 
$8g + 8$) is injective. Similarly, we have a restriction map $r':\Gamma (C,  h\otimes E\otimes \iota ^*E)$ (of 
dimension $8g+4 + 4(1-g) = 4g+8$ into $W(h\otimes E\otimes E)$, which is also injective. 

The co-kernels of $r$ and $r'$ are respectively $H^1(C, h^{-g-1}\otimes E\otimes \iota^*E)$ and 
$H^1(C, h^{-g}\otimes E\otimes \iota ^*E)$. Since the dual of $E$ is isomorphic to $E\otimes \eta ^{-1}$,
it follows that the dual of $h^{-g-1}\otimes E\otimes \iota^*E$ is $h^{g + 1}\otimes 
E\otimes \iota ^*(E)\otimes \eta ^{-1} \otimes \iota ^*(\eta ^{-1}) = h^{-g+2}\otimes E\otimes \iota ^*(E)$. 
Hence we conclude that the cokernel of $r$ 
is dual to $\Gamma (C, h\otimes E\otimes \iota ^*(E))$.  Similarly the cokernel of $r'$ is 
$H^1(C, E\otimes \iota ^*(E) \otimes h^{-g})$ which is dual to $\Gamma (C, E\otimes \iota ^*(E))$. 

Note that all these vector spaces are acted on by $\iota $ and the maps respect 
the $\iota$-action so that we may replace them by the respective $-1$-eigenspaces to get two dual short exact sequences:
\begin{enumerate}
\item[1)] $0 \to \Gamma (C, E\otimes\iota ^*E)^- \to W(\eta ) \to (\Gamma (C, h\otimes E\otimes \iota^*E)^-)^*\to 0$.
\item[2)] $0 \to \Gamma (C, h\otimes E\otimes \iota ^*E)^- \to W(\xi ) \to (\Gamma (C, E\otimes \iota ^*E)^-)^*\to 0$.
\end{enumerate}

By Atiyah-Bott fixed point theorem, the trace of the action $\iota $ on $\Gamma (X, E\otimes \iota ^*E)$ is $(2g +2)$.
Hence dimension of $\Gamma (C, E\otimes\iota ^*E)^-$ is $1/2(4g - 2g -2) = g - 1$. 

The pencil maps $\Gamma (C, E\otimes\iota ^*E)^-$ into $\Gamma (C, h \otimes E\otimes\iota ^*E)^-$ 
and hence maps to 0 into its dual. In other words, we have associated to every stable
 vector bundle $E$ of rank 2 and determinant $\eta $, a subspace of $W(\eta )$ of dimension $g-1$, 
 which is isotropic to all the quadrics of the pencil.

\begin{theorem}
\label{app-main}
There is a natural isomorphism of the moduli space $M(2, \eta )$ of stable bundles of rank 2 and determinant $\eta $ with the space $I_{g-1}$ 
of $(g-1)$-dimensional isotropic subspaces of the associated pencil, under the assumption that $C$ has a rational point.s a rational point.
\end{theorem}

\begin{proof}
It is proved in \cite{De-Ra} that over the algebraic closure of $k$, this association gives an isomorphism of the moduli space of 
vector bundles of rank $2$ and determinant $\eta $ with the space of $(g-1)$-dimensional
 subspaces isotropic to the pencil. Hence the above map, defined over $k$, is also an isomorphism.
\end{proof}

If $\alpha \in J_2$ and we replace $E$ by $\alpha \otimes E$, the action of $\iota $ on $(\alpha \otimes E) \otimes \iota ^*(\alpha \otimes E)
 = \alpha \otimes \iota ^*\alpha \otimes E\otimes \iota^*E$ is the same as on $E\otimes \iota^*E$ 
since its action on $\alpha \otimes \iota^*\alpha $ is trivial. Hence $\Gamma (E\otimes \iota ^*E)^-$ and $\Gamma 
((\alpha \otimes E)\otimes \iota ^*(\alpha \otimes E))^-$ are isomorphic. On the other hand, as we have seen in \ref{the-maps-s}, 
the trivialisation of  $\alpha \otimes \iota^*\alpha $ restricts (over the algebraic closure) to $-1$ on a subset of $W$ of even cardinality and $+1$ in the complement. Hence we have 

\begin{prop}
\label{actions}
The isomorphism in \ref{app-main} is compatible with
the action of $J_2$ on  $M(2, \eta )$ and that on $I_{g-1}$.
\end{prop}

\begin{remark} 
In a similar manner, one can  interpret $I_d$ for other $d\leq g - 2$ in terms of bundles on $C$, on the lines of the results of \cite{Ra1}, over the field $k$ itself.
\end{remark}

\providecommand{\bysame}{\leavevmode\hbox to3em{\hrulefill}\thinspace}


\begin{thebibliography}{10}
 
 
\bibitem{Beauville} A. Beauville,  Y. Laszlo, Ch. Sorger, 
{\em The Picard group of the moduli of G-bundles on a curve}. Compositio math. \textbf{112}, 183-216 (1998).
 
\bibitem{H-B-N} R. Brauer, H. Hasse and E. Noether, { \em Beweis eines Hauptsatzes in der Theorie der Algebren,}
J. Reine Angew. Math. \textbf{167} (1931), 399-404.

 \bibitem{Manjul}  Manjul Bhargava,   { \em Most hyperelliptic curves over Q have no rational points}, 
 archiv.org/abs/1308.0395vl,  2013

 \bibitem{Ca} J.W.S. Cassels,  { \em Arithmetic on a curve of genus one. (IV) Proof of the Hauptvermutung},
Proc. London Math. Soc. 46 (1962), 259-296.

\bibitem{CT2005}  J.-L. Colliot-Th\'el\`ene,  \emph{Un th\'eor\`eme de finitude pour le groupe de Chow des
z\'ero-cycles d'un groupe alg\'ebrique lin\'eaire sur un corps $p$-adique}, Invent. math. {\bf 159} (2005), 589--606. 

\bibitem{CT88}  J.-L. Colliot-Th\'el\`ene, \emph{Surfaces rationnelles fibr\'ees en coniques de degr\'e 4},  
  S\'eminaire de th\'eorie des nombres de Paris 88-89,  Progr. Math., t. 91 (1990), 43-55.  


\bibitem{CTSaSwD87} 
 Colliot-Th\'el\`ene,  Jean-Louis;  Sansuc, Jean-Jacques;  Swinnerton-Dyer, Peter, 
 \emph{Intersections of two quadrics and Ch\^atelet surfaces II},    J. Reine Angew. Math. {\bf 374}(1987), 72-168.
 
 
\bibitem{CTSk}  Jean-Louis Colliot-Th\'el\`ene , Alexei N. Skorobogatov, \emph{The Brauer-Grothendieck Group}, 
Results in Mathematics and Related Areas. 3rd Series. A Series of Modern Surveys in Mathematics,  {\bf 71} Springer, Cham,  2021.

\bibitem{DJ} A. J. de Jong, { \em The period-index problem for the Brauer group of an algebraic surface}, 
Duke Math. J. 123 (2004), no. 1, 71-94.

\bibitem{De-Ra} U. Desale, S.  Ramanan { \em Classification of vector bundles of rank $2$ on
 hyperelliptic curves} Inv.Math 38, 161--185 (1976). 

 \bibitem{Ek89}  Torsten Ekedahl, \emph{ An  effective version of Hilbert's irreducibility theorem},  
 S\'eminaire de Th\'eorie des Nombres, Paris 1988–1989, 241–249, Progr. Math., 91

\bibitem{LG}  Luc Gauthier,  { \em Footnote to a footnote of Andr{\'e} Weil},   Univ. de Politec. Torino, Rend. Sem. Math.
{\bf 14}, 325-328 (1955)
 
\bibitem{G} P. Gille, { \em Cohomologie galoisienne des groupes quasi-d\'eploy\'es sur 
des corps de dimension cohomologique $\leq $ 2.}  Compositio Math. 125 (2001), no. 3, 283–325.

\bibitem{Giraud} J. Giraud, { \em Cohomologie non ab\'elienne.} 
Springer-Verlag, Berlin, 1971. Die Grundlehren der mathematischen Wissenschaften, Band 179.

\bibitem{HHK} \text{D. Harbater, J. Hartmann and D.
    Krashen}, \textit{Applications of patching to quadratic  forms and
    central simple algebras}, Invent.  Math.  {\bf 178}
  (2009),  231--263.  

  
\bibitem{KMRT}  \text{M.-A. Knus, A.S. Merkurjev, M. Rost and J.-P. Tignol},
 \textit{The Book of Involutions}, A.M.S, Providence RI, 1998. 
 
\bibitem{Lb} M. Lieblich, { \em Twisted sheaves and the period-index problem.}  Compos. Math.  144  (2008),  no. 1, 131.

\bibitem{Lb2} M. Lieblich, { \em Moduli of twisted sheaves.}  Duke Math. J.  138  (2007),  no. \textbf{1}, 23--118.

\bibitem{Lb3}  M. Lieblich, { \em Arithmetic aspects of moduli spaces of sheaves on curves.} 
 Grassmannians, moduli spaces and vector bundles,  95120, Clay Math. Proc., 14, Amer. Math. Soc., Providence, RI, 2011. 

\bibitem{Lb-P-S} M. Lieblich, R. Parimala, V. Suresh, 
{ \em Colliot-Th\'el\`ene's conjecture and finiteness of u-invariants},  Math. Annalen \textbf{360} (2014), 1-22.  

\bibitem{Milne} J.S. Milne, Jacobian Varieties, {\it Arithmetic Varieties}, 167-212, Springer, 1986.

\bibitem{NR} M.S. Narasimhan, S. Ramanan, { \em Moduli of vector bundles on a compact Riemann surface.} 
 Ann. of Math. (2)  \textbf{89} (1969), 14--51.

\bibitem{N1R1} M.S. Narasimhan, S. Ramanan, { \em Vector bundles on curves},  
 1969 Algebraic Geometry (Internat. Colloq., Tata Inst. Fund. Res., Bombay, 1968) pp. 335–346

\bibitem{Na-Ra} M.S. Narasimhan, S. Ramanan, { \em Geometry of Hecke cycles. I. 
C. P. Ramanujama tribute}, p. 291345, Tata Inst. Fund. Res. Studies in Math., \textbf{8}, Springer, Berlin-New York, 1978.

\bibitem{ON} C. O’Neil, { \em  The period-index obstruction for elliptic curves}, J. Number Theory \textbf{97} (2002),
329-339.

 \bibitem{Ra} S. Ramanan, { \em The  Moduli  Spaces  of Vector  Bundles  over an  Algebraic  Curve}, Math. Ann. 200, 69--84 (1973) 
 
\bibitem{Ra1} S. Ramanan, Orthogonal and Spin bunldes over hyperellipitic curves, {\it Proc. Ind. Acad. Sci.}{\bf 90}, 151-166 (1981).

\bibitem{Reid} M. Reid. { \em The complete the intersection  of 2  or  more   quadrics},    thesis  Cambridge,  1972.  
 
\bibitem{P-Sj} R. Parimala, R. Sujatha, \emph{Hasse principle for Witt groups of function fields with special reference to elliptic curves.} 
With an appendix by J.-L. Colliot-Th\'el\`ene.  Duke Math. J.  85  (1996),  no. \textbf{3}, 555--582.

\bibitem{PSIMRN} R. Parimala, V. Suresh, \emph{Degree 3 cohomology of function fields of surfaces},  
Int. Math. Res. Not. IMRN {\bf 14}(2016,  4341-4374.


  \bibitem{Sal88} P. Salberger,  \emph{Zero-cycles on rational surfaces over number  fields},   Invent. math. {\bf 91} (1988), no.
3, 505-524.
 
\bibitem{saltman} D. Saltman, { \em Division algebras over p-adic curves.}  J. Ramanujan Math. Soc.  12  (1997),  no. 1, 25--47. 

 
\bibitem{Serre} J.-P. Serre,  { \em Galois cohomology}  (Springer, New York, 1997).

\bibitem{SW} A. Shankar and X. Wang, 
{ \em Rational points on hyperelliptic curves having a marked non-Weierstrass point}, Compositio Math. 154 (2018), 188-222.

\bibitem{T} D. Tao,  { \em  A variety associated to an algebra with involution.} J. Algebra 168 (1994),  479-520.

\bibitem{Tj} A. Tjurin,  { \em  Analogs of Torelli's theorem for multi-dimensional vector  bundles over an 
arbitrary  algebraic  curve.}  Izv.  Akad.  Nauk  SSSR,  Set.  Mat.  Tom   34  (2), 338--365 
(1970),  Math.  U S S R    Izvestija  Vol. 4  (2),  343--370  (1970)

\bibitem{Wang} X. Wang,  { \em Maximal linear spaces contained in the base loci of pencils of quadrics}, 
Algebraic Geometry 5 (3) (2018) 359-397. 

\bibitem{OW}  O. Wittenberg,  { \em Principe de Hasse pour les intersections de deux quadriques},
C. R. Math. Acad. Sci. Paris 342 (2006), no. 4, 223-227.



\end{thebibliography}
\end{document}